\documentclass[a4paper,11pt]{amsart}
\usepackage{amssymb}
\usepackage{amsmath}
\usepackage{cancel}
\usepackage{mathrsfs}
\usepackage{color}
\usepackage{caption}
\usepackage{subcaption}
\usepackage{hyperref}
\usepackage{backref}
\usepackage{enumerate}
\usepackage{mathtools}
\usepackage{comment}
\usepackage{color}
\usepackage{multirow}
\usepackage{bm}

\pdfoutput=1
\usepackage[a4paper,margin=3cm]{geometry}
\DeclareSymbolFont{script}{U}{eus}{m}{n}
\DeclareSymbolFontAlphabet{\amathscr}{script}
\DeclareMathSymbol{\Wedge}{0}{script}{"5E}
\DeclareMathAlphabet{\mathrmsl}{OT1}{cmr}{m}{sl}

\theoremstyle{definition}

\theoremstyle{definition}
\newtheorem{lemma}{Lemma}[section]
\newtheorem{proposition}[lemma]{Proposition}
\newtheorem{theorem}[lemma]{Theorem}
\newtheorem{corollary}[lemma]{Corollary}
\newtheorem{definition}[lemma]{Definition}
\newtheorem{remark}[lemma]{Remark}
\newtheorem{example}[lemma]{Example}

\DeclareMathOperator{\grad}{grad}
\DeclareMathOperator{\Aut}{Aut}
\DeclareMathOperator{\Ric}{Ric}
\DeclareMathOperator{\Scal}{Scal}
\DeclareMathOperator{\vol}{vol}

\DeclareMathOperator{\Fut}{Fut}

\DeclareMathOperator{\Ding}{Ding}

\makeatletter

\@addtoreset{equation}{section}
\makeatother

\allowdisplaybreaks

\makeatletter
\@namedef{subjclassname@2020}{\textup{2020} Mathematics Subject Classification}
\makeatother


\thanks{The second author was supported by JSPS KAKENHI Grant Numbers JP21K03234, 26K06790.}

\begin{document}

\title[Mabuchi solitons and Mabuchi constants on Fano admissible manifolds]{Mabuchi solitons and Mabuchi constants \\ on Fano admissible manifolds}

\author[Shotaro \textsc{Murayama}]{Shotaro \textsc{Murayama}}
\address{Department of Mathematics \\ Graduate School of Science \\Tokyo University of Science}
\email{1126708@ed.tus.ac.jp}

\author[Yasufumi \textsc{Nitta}]{Yasufumi \textsc{Nitta}}
\address{Department of Mathematics \\ Faculty of Science Division II \\Tokyo University of Science}
\email{nitta@rs.tus.ac.jp}


\begin{abstract}
In this paper, we study the existence of Mabuchi solitons on admissible manifolds as defined by Apostolov--Calderbank--Gauduchon--T\o nnesen-Friedman. 
We prove that a Fano admissible manifold admits a Mabuchi soliton if and only if the Mabuchi constant is less than 1. 
We also provide an explicit formula for the Mabuchi constant on Fano admissible manifolds, which generalizes that of Mabuchi. 
Using this formula, we completely determine the existence and non-existence of Mabuchi solitons on Fano admissible manifolds over the complex projective space $\mathbf{P}^{n}$. 
\end{abstract}

\maketitle

\section{Introduction}\label{sec1_intro}

\subsection{Background}

A \emph{Mabuchi soliton}, introduced by Mabuchi in \cite{Ma01}, is a generalization of a K\"ahler--Einstein metric. 
Let $X$ be an $n$-dimensional Fano manifold, and let $G$ be a maximal compact subgroup of $\Aut^{0}(X)$. 
Let $\omega$ be a $G$-invariant K\"ahler metric on $X$ representing $2\pi c_{1}(X)$. 
By the $dd^{c}$-lemma, there exists a unique function $h_{\omega} \in C^{\infty}(X, \mathbf{R})^{G}$ satisfying 
\[
\Ric(\omega) - \omega = \frac{1}{2}dd^{c}h_{\omega},\quad \int_{X}(1-e^{h_{\omega}})\frac{\omega^{n}}{n!} = 0. 
\]
This function $h_{\omega}$ is called the \emph{Ricci potential} of $\omega$. 

A K\"ahler metric $\omega \in 2\pi c_{1}(X)^{G}$ is called a \emph{Mabuchi soliton} if the gradient 
\[
\grad_{\omega}(1-e^{h_{\omega}}) 
\]
with respect to $\omega$ is a real holomorphic vector field on $X$. 
It is immediate that any $G$-invariant K\"ahler--Einstein metric is a Mabuchi soliton. 
Moreover, a Mabuchi soliton is a K\"ahler--Einstein metric if and only if the Futaki invariant of $X$ vanishes (see \cite[Section 1]{Ma01}). 

One motivation for studying Mabuchi solitons comes from the moment map interpretation of the Ding functional. 
Let $\mathcal{H}(X)^{G}$ be the space of all $G$-invariant K\"ahler metrics representing $2\pi c_{1}(X)$. 
Then, Donaldson showed in \cite{D17} that the map
\[
\mathcal{H}(X)^{G} \ni \omega \mapsto 1-e^{h_{\omega}} \in C^{\infty}(X, \mathbf{R})
\]
may be viewed as a moment map, whose zeros are precisely K\"ahler--Einstein metrics. 
In analogy with the Calabi energy in the scalar curvature moment map picture, one can consider the Ding energy
\[
\Ding(\omega) \coloneqq \frac{1}{V}\int_{X}(1-e^{h_{\omega}})^{2}\frac{\omega^{n}}{n!},\quad V \coloneqq \int_{X}\frac{\omega^{n}}{n!}. 
\]
Yao showed that the critical points of this functional are precisely Mabuchi solitons (\cite[Theorem 2.6]{Y22b}). 
Thus, in the Ding-theoretic setting, Mabuchi solitons play a role analogous to that of Calabi's extremal K\"ahler metrics (\cite{C82}) in the scalar curvature setting. 
Moreover, the existence of Mabuchi solitons is related to relative Ding stability through a Yau--Tian--Donaldson type correspondence (\cite{HL23}). These features make Mabuchi solitons natural canonical metrics on Fano manifolds for which the Futaki invariant does not necessarily vanish.

An obstruction to the existence of Mabuchi solitons is given by the \emph{Mabuchi constant}. 
Let $P_\omega$ be the space of Killing potentials with respect to $\omega$, and let $\Pi_\omega^{G} \colon C^{\infty}(X,\mathbf{R}) \to P_{\omega}$ be the $L^{2}$-projection. 
Then Mabuchi showed that the constant 
\[
M_{X} \coloneq \max_{X}\Pi_{\omega}^{G}(1-e^{h_{\omega}})
\]
is independent of the choice of $\omega$ (\cite[Section 2]{Ma01}). 
This $M_{X}$ is called the \emph{Mabuchi constant} of $X$. 
The Mabuchi constant satisfies $M_{X} < 1$ if $X$ admits a Mabuchi soliton. 
In this paper, we study a special class of manifolds, called admissible manifolds (\cite[Section 1]{ACGTF08-3}), and prove that the converse holds for Fano admissible manifolds. 
(We also point out that a similar equivalence holds for toric Fano manifolds; see \cite{Y22b} for details) 
Moreover, we give an explicit formula for the Mabuchi constant of Fano admissible manifolds, generalizing \cite[Theorem 6.5]{Ma01}. 
Using this, we completely determine the existence and non-existence of Mabuchi solitons on admissible manifolds over a complex projective space $\mathbf{P}^{n}$. 

\subsection{Main results}
An \emph{admissible manifold} is a projective bundle $\mathbf{P}_{Y}(E_{0} \oplus E_{\infty}) \to Y$ satisfying the following conditions: 
\begin{itemize}
\item $Y$ is covered by a product $\tilde{Y}\coloneq\prod_{a \in \mathcal{A}}Y_{a}$ ($\mathcal{A} \subset \mathbf{Z}_{>0}, \#\mathcal{A}<\infty$) of simply-connected compact K\"ahler manifolds $(Y_{a}, g_{a}, \omega_{a})$ of real dimension $2d_{a}$, with $(g_{a}, \omega_{a})$ being pullbacks of tensors on $Y$. 
\item $E_{0}$ and $E_{\infty}$ are holomorphic projectively-flat Hermitian vector bundles over $Y$ of ranks $d_{0} + 1$ and $d_{\infty} + 1$, respectively, satisfying 
\[
\frac{c_{1}(E_{\infty})}{d_{\infty}+1} - \frac{c_{1}(E_{0})}{d_{0}+1} = \left[\frac{\omega_{Y}}{2\pi}\right],\quad \omega_{Y} \coloneq \sum_{a \in \mathcal{A}} \varepsilon_{a}\omega_{a}, 
\]
where $\varepsilon_{a} \in \{-1,1\}$ ($a \in \mathcal{A}$). 
\end{itemize}
Assume further that each $\omega_{a}$ is a Ricci positive K\"ahler--Einstein metric with $\Ric(\omega_{a})=\varepsilon_{a}s_{a}\omega_{a}$ for $s_{a} \in \mathbf{R}$. 
Then $X = \mathbf{P}_{Y}(E_{0} \oplus E_{\infty})$ is a Fano manifold if and only if 
\begin{align*}
\begin{cases*}
s_{a} > d_0+1 & \text{(if $\varepsilon_{a}=1$)}, \\
s_{a} < -(d_{\infty} + 1)&\text{(if $\varepsilon_{a}=-1$)}
\end{cases*}
\end{align*}
for any $a \in \mathcal{A}$ (\cite[Theorem 3.1]{ACGTF08-4}). 
Such an $X$ is called a \emph{Fano admissible manifold}. 

Our first result shows that, in Fano admissible manifolds, the existence of a Mabuchi soliton can be completely determined by the Mabuchi constant. 

\begin{theorem}[{cf. Corollary \ref{existence_adm_M-sol}}]\label{Main_thm_0}
A Fano admissible manifold $X$ admits a Mabuchi soliton if and only if $M_{X} < 1$. 
\end{theorem}

The proof of Theorem \ref{Main_thm_0} is obtained as a consequence of a more general existence theory for \emph{$v$-solitons}, and will be given in Section \ref{sec:ex_v-sol_Fano_adm}. 
In general, the existence of a $v$-soliton is obstructed by the $v$-Futaki invariant (cf. Proposition \ref{v-sol_vanish_vFut}). 
For Fano admissible manifolds, however, we show that the existence of an admissible $v$-soliton is in fact equivalent to the vanishing of the $v$-Futaki invariant (see Theorem \ref{existence_adm_v-sol}). On the other hand, the $v$-Futaki invariant corresponding to Mabuchi solitons always vanishes. 
Therefore, \ref{Main_thm_0} follows immediately from Theorem \ref{existence_adm_v-sol}.

By Theorem \ref{Main_thm_0}, we can determine the existence of Mabuchi solitons on Fano admissible manifolds in terms of the Mabuchi constant. 
We next compute the Mabuchi constant for Fano admissible manifolds. 
Let 
\[
\Omega \coloneqq 2\pi c_{1}(X),\quad x_{a} \coloneq \frac{d_{0} + d_{\infty} + 2}{2s_{a}+d_{\infty} - d_{0}}\quad (a \in \mathcal{A}), 
\]
and define the \emph{characteristic polynomial} $p_{\Omega}$ of $\Omega$ by 
\[
p_{\Omega}(x) \coloneq (1+x)^{d_{0}}(1-x)^{d_{\infty}}\prod_{a \in \mathcal{A}}\left(\frac{\varepsilon_{a}}{x_{a}}+\varepsilon_{a} x\right)^{d_{a}} 
\]
(cf. Section \ref{sec3_characteristic_poly}). 
For $i \in \{0, 1, 2\}$, set 
\[
b_{i} \coloneq \int_{-1}^{1}x^{i}p_{\Omega}(x)\,dx. 
\]

In the case $d_{0} = d_{\infty} = 0$, Mabuchi obtained an explicit formula for $M_{X}$ in terms of $b_{0}$, $b_{1}$, $b_{2}$ (\cite[Theorem 6.5]{Ma01}). 
We extend this to general Fano admissible manifolds as follows.

\begin{theorem}[{cf. Theorem \ref{M-const_proj_bdle}}]\label{Main_thm_1}
Let $w \coloneqq \dfrac{d_{0} - d_{\infty}}{d_{0} + d_{\infty} + 2}$. 
Then 
\begin{align*}
M_{X} 
= \frac{b_{0}|b_{1} - wb_{0}| - b_{1}(b_{1} - wb_{0})}{b_{0}b_{2} - b_{1}^{2}} 
= 1 + \frac{b_{0}(|b_{1} - wb_{0}| - (b_{2} - wb_{1}))}{b_{0}b_{2}-b_{1}^{2}}. 
\end{align*}
\end{theorem}

Finally, we consider admissible manifolds over the complex projective space $\mathbf{P}^{n}$ and discuss the existence of Mabuchi solitons on them.
Any Fano admissible manifold $X$ over $\mathbf{P}^{n}$ can be written as 
\[
X = \mathbf{P}_{\mathbf{P}^n}(\mathcal{O}^{\oplus(d_{0} + 1)} \oplus \mathcal{O}(k)^{\oplus(d_{\infty} + 1)})
\]
for some $k \in \mathbf{Z}$ satisfying $k \ge 1$ and $n+1>k(d_{0}+1)$. 
Moreover, the existence of a Mabuchi soliton on $X$ is equivalent to the condition $M_{X} < 1$ (see Corollary \ref{existence_adm_M-sol}). 
Using Theorem \ref{Main_thm_1}, we prove the following theorem. 

\begin{theorem}[{cf. Theorem \ref{M-sol_adm/P^n}}]\label{Main_thm_2}
The Fano admissible manifold $X$ admits a Mabuchi soliton if and only if 
\[
(k, d_{\infty})=(1,0)\quad \text{or}\quad (n, k, d_{0}, d_{\infty}) = (1, 1, 0, 1). 
\]
In all other cases, $M_{X} > 1$. 
\end{theorem}

Theorem \ref{Main_thm_2} can also be interpreted from the viewpoint of stability. 
By the Yau--Tian--Donaldson type correspondence for Mabuchi solitons, established by Han and Li \cite{HL23}, the existence of a Mabuchi soliton on a Fano manifold is equivalent to \emph{uniform relative Ding stability} (see also \cite{Y22b, Nak19, Y22a, Hi23}).
Moreover, by \cite[Theorem 1]{Y22a}, the Mabuchi constant of any relatively Ding semistable Fano manifold is bounded above by 1. 
Combining these results with Theorem \ref{Main_thm_2}, we obtain the following. 

\begin{corollary}\label{D-stab_adm/P^n}
$X$ is uniformly relatively Ding stable if and only if 
\[
(k, d_{\infty})=(1,0)\quad \text{or}\quad (n, k, d_{0}, d_{\infty}) = (1, 1, 0, 1). 
\]
Otherwise, $X$ is relatively Ding unstable. 
\end{corollary}

\subsection{Organization of the paper}
We start by reviewing basic notions on Mabuchi solitons in Section \ref{sec2_M-sol}. 
The Mabuchi constant is also introduced. 
In Section \ref{sec3_adm_mfds}, following \cite{ACGTF08-3}, we introduce the notion of admissible manifolds. 
After introducing admissible manifolds, we discuss admissible K\"ahler classes and admissible K\"ahler metrics, which are the main tools of this paper, as well as the characteristic polynomial associated with an admissible K\"ahler class. 
In Section \ref{sec: Fano_admissible_manifold}, we focus on those admissible manifolds that are Fano. 
We give an explicit formula for the Ricci potential of admissible K\"ahler metrics on Fano admissible manifolds in Section \ref{sec:Ricci_potential} (cf. Proposition \ref{adm_Ric_pot}). 
This is a refinement of the result of Maschler and T\o nnesen-Friedman \cite[Corollary 3.3]{MTF11}. 
Using Proposition \ref{adm_Ric_pot}, we obtain Proposition \ref{mommap_normalization} for the normalization of the moment map, which plays an important role in the proof of Theorem \ref{Main_thm_1}. 
In Section \ref{sec:v-sol_Fano_adm}, we introduce the notion of $v$-solitons on general Fano manifolds and explain that Mabuchi solitons form a special class of $v$-solitons. 
We also describe the $v$-Futaki invariant, which is an obstruction to the existence of $v$-solitons, and show that, for Fano admissible manifolds, the existence of an admissible $v$-soliton is equivalent to the vanishing of the $v$-Futaki invariant (cf. Theorem \ref{existence_adm_v-sol}). 
From Theorem \ref{existence_adm_v-sol} and the fact that the $v$-Futaki invariant corresponding to Mabuchi solitons always vanishes (see Remark \ref{M-sol_vFut_vanish}), we obtain Theorem \ref{Main_thm_0}.
In Section \ref{prf_Main_thm_1}, we study the Mabuchi constant for Fano admissible manifolds and give its explicit formula (Theorem \ref{Main_thm_1}). 
Finally, in Section \ref{sec_computation}, we prove Theorem \ref{Main_thm_2}. 
The proof proceeds by computing the Mabuchi constant for Fano admissible manifolds over $\mathbf{P}^{n}$ using Theorem \ref{Main_thm_1}, and applying Theorem \ref{Main_thm_0} to determine the existence of Mabuchi solitons.

\section{Mabuchi solitons and the Mabuchi constant}\label{sec2_M-sol}
We first recall the definitions of Mabuchi solitons and the Mabuchi constant. 
Let $X$ be an $n$-dimensional Fano manifold, and let $G$ be a maximal compact subgroup of $\mathrm{Aut}^{0}(X)$. 
Let $\omega$ be a $G$-invariant K\"ahler metric on $X$ representing $2\pi c_{1}(X)$, and denote by $h_{\omega}$ the Ricci potential of $\omega$. 

\begin{definition}
Let $\omega \in 2\pi c_{1}(X)^{G}$ be a $G$-invariant K\"ahler metric on $X$. 
We say that $\omega$ is a \emph{Mabuchi soliton} if the gradient 
\[
\grad_{\omega}(1-e^{h_{\omega}}) 
\]
with respect to $\omega$ is a real holomorphic vector field on $X$. 
\end{definition}

\begin{example}\label{KE_M-sol}
Let $\omega \in 2\pi c_{1}(X)^{G}$ be a $G$-invariant K\"ahler--Einstein metric on $X$: $\Ric(\omega) = \omega$. 
Then $h_{\omega} = 0$ and hence 
\[
\mathrm{grad}_{\omega}(1-e^{h_{\omega}}) = \mathrm{grad}_{\omega}0 = 0. 
\] 
Therefore, $\omega$ is a Mabuchi soliton. 
\end{example}

\begin{example}[{\cite[Example 5.8]{Ma01}}]
Let $X \coloneqq \mathbf{P}_{\mathbf{P}^{n}}(\mathcal{O}_{\mathbf{P}^{n}}\oplus\mathcal{O}_{\mathbf{P}^{n}}(1))$ be the Fano manifold obtained by blowing up $\mathbf{P}^{n+1}$ at a point. 
Then $X$ admits a Mabuchi soliton but does not admit a K\"ahler--Einstein metric. 
\end{example}

Next, we explain the Mabuchi constant, which is known as an obstruction to the existence of Mabuchi solitons. 
Let $P_{\omega}$ denote the space of Killing potentials with respect to $\omega$, i.e., the space of real-valued smooth functions $f \in C^{\infty}(X, \mathbf{R})$ such that $\grad_{\omega}f$ is a real holomorphic vector field on $X$. 
Let $\Pi_{\omega}^{G} \colon C^{\infty}(X, \mathbf{R})^{G} \to P_{\omega}$ be the $L^{2}$-projection with respect to $\omega$. 
Then a $G$-invariant K\"ahler metric $\omega \in 2\pi c_{1}(X)^{G}$ is a Mabuchi soliton if and only if $\Pi_{\omega}^{G}(1-e^{h_{\omega}}) = 1-e^{h_{\omega}}$. 
Moreover, Mabuchi showed in \cite{Ma01} that the constant 
\[
\max_{X}\Pi_{\omega}^{G}(1-e^{h_{\omega}})
\]
is independent of the choice of $\omega \in 2\pi c_{1}(X)^{G}$. 

\begin{definition}
We define $M_{X} \coloneqq \max_{X}\Pi_{\omega}^{G}(1-e^{h_{\omega}})$ and call it the \emph{Mabuchi constant} of the Fano manifold $X$. 
\end{definition}

The following examples show that the Mabuchi constant serves as an obstruction to the existence of K\"ahler--Einstein metrics and Mabuchi solitons. 

\begin{example}\label{KE_M-const}
Assume that $2\pi c_{1}(X)^{G}$ contains a K\"ahler--Einstein metric $\omega$. 
Then $h_{\omega} = 0$, and hence 
\[
\Pi_{\omega}^{G}(1-e^{h_{\omega}}) = \Pi_{\omega}^{G}0 = 0. 
\]
Therefore, we have ${\displaystyle M_{X} = \max_{X}\Pi_{\omega}^{G}(1-e^{h_{\omega}}) = 0}$. 
\end{example}

\begin{example}[{\cite[Theorem 3.1]{Ma01}}]\label{mc_obst}
Assume that $2\pi c_{1}(X)^{G}$ contains a Mabuchi soliton $\omega$. 
Then 
\[
\Pi_{\omega}^{G}(1-e^{h_{\omega}}) = 1-e^{h_{\omega}} < 1 
\]
holds on $X$. 
Hence, ${\displaystyle M_{X} = \max_{X}\Pi_{\omega}^{G}(1-e^{h_{\omega}}) < 1}$. 
\end{example}

\section{Admissible manifolds}\label{sec3_adm_mfds}
\subsection{Definition of admissible manifolds}
In this section, following the description in \cite{ACGTF08-3}, we explain the definition of admissible manifolds. For details, we refer the reader to~\cite{ACGTF08-3}. 

\begin{definition}
An \emph{admissible manifold} is a projective bundle $\mathbf{P}_{Y}(E_{0} \oplus E_{\infty}) \to Y$ satisfying the following conditions: 
\begin{itemize}
\item $Y$ is covered by a product $\tilde{Y}\coloneq\prod_{a \in \mathcal{A}}Y_{a}$ ($\mathcal{A} \subset \mathbf{Z}_{>0}, \#\mathcal{A}<\infty$) of simply-connected compact K\"ahler manifolds $(Y_{a}, g_{a}, \omega_{a})$ of real dimension $2d_{a}$, with $(g_{a}, \omega_{a})$ being pullbacks of tensors on $Y$. 
\item $E_{0}$ and $E_{\infty}$ are holomorphic projectively-flat Hermitian vector bundles over $Y$ of ranks $d_{0} + 1$ and $d_{\infty} + 1$, respectively, satisfying 
\begin{align}\label{adm_conn_curv}
\frac{c_{1}(E_{\infty})}{d_{\infty}+1} - \frac{c_{1}(E_{0})}{d_{0}+1} = \left[\frac{\omega_{Y}}{2\pi}\right],\quad \omega_{Y} \coloneq \sum_{a \in \mathcal{A}} \varepsilon_{a}\omega_{a}, 
\end{align}
where $\varepsilon_{a} \in \{-1,1\}$ ($a \in \mathcal{A}$). 
\end{itemize}
\end{definition}
Let $X = \mathbf{P}_{Y}(E_{0} \oplus E_{\infty})$ be an admissible manifold. 
By the condition \eqref{adm_conn_curv} above, there exist Hermitian metrics on $E_{0}$ and $E_{\infty}$ such that the curvatures of their Chern connections are given by $\Omega_{0} \otimes \mathrm{id}_{E_{0}}$, $\Omega_{\infty} \otimes \mathrm{id}_{E_{\infty}}$, and satisfy $\Omega_{\infty} - \Omega_{0} = \omega_{Y}$. 
Let $(g_{0}, \omega_{0})$ and $(g_{\infty}, \omega_{\infty})$ denote the fiberwise Fubini--Study metrics of $\mathbf{P}(E_{0})$ and $\mathbf{P}(E_{\infty})$. 
We normalize them so that their scalar curvatures are $2d_{0}(d_{0}+1)$ and $2d_{\infty}(d_{\infty}+1)$, respectively. 
We also set $\varepsilon_{0} \coloneqq 1$ and $\varepsilon_{\infty} \coloneqq -1$. 

We summarize below the notation for admissible manifolds used in this paper: 
\begin{enumerate}[(1)]
\item We set $\hat{\mathcal{A}} \coloneqq \{a \in \mathcal{A} \cup \{0, \infty\} \mid d_{a} > 0\}$. 
\item Let $\tilde{E}_{0}$ and $\tilde{E}_{\infty}$ be the pullbacks of $E_{0}$ and $E_{\infty}$ to $\tilde{Y}$, respectively. 
Since $\tilde{Y}$ is simply-connected, they can be written as 
\[
\tilde{E}_{0} = \mathcal{E}_{0} \otimes \mathcal{O}^{\oplus (d_{0}+1)},\quad \tilde{E}_{\infty} = \mathcal{E}_{\infty} \otimes \mathcal{O}^{\oplus (d_{\infty}+1)} 
\]
for some holomorphic line bundles $\mathcal{E}_{0}$ and $\mathcal{E}_{\infty}$ over $\tilde{Y}$. 
Let $\mathcal{L} \coloneqq \mathcal{E}_{0}^{-1} \otimes \mathcal{E}_{\infty}$. 
Then $\mathcal{L}$ is a holomorphic line bundle on $\tilde{Y}$, and can be written as 
\[
\mathcal{L} = \bigotimes_{a \in \mathcal{A}}\mathcal{L}_{a}, 
\]
where each $\mathcal{L}_{a}$ is a holomorphic line bundle over $Y_{a}$ satisfying 
\[
c_{1}(\mathcal{L}_{a}) = \left[\frac{\varepsilon_{a}\omega_{a}}{2\pi}\right]. 
\]
\item Let $e_{0} \coloneqq \mathbf{P}_{Y}(E_{0} \oplus 0)$, $e_{\infty} \coloneqq \mathbf{P}_{Y}(0 \oplus E_{\infty})$. 
Their universal covers are, respectively, $Y_{0} \times \tilde{Y} \coloneqq \mathbf{P}^{d_{0}} \times \tilde{Y}$ and $\tilde{Y} \times Y_{\infty} \coloneqq \tilde{Y} \times \mathbf{P}^{d_{\infty}}$. 
\item The blow-up of $X$ along $e_{0} \cup e_{\infty}$ is given by $\hat{X} \coloneqq \mathbf{P}_{\hat{Y}}(\mathcal{O} \oplus \hat{\mathcal{L}}) \to \hat{Y}$, where $\hat{Y} \coloneqq \mathbf{P}_{Y}(E_{0}) \times_{Y} \mathbf{P}_{Y}(E_{\infty}) \to Y$ and $\hat{\mathcal{L}} \coloneqq \mathcal{O}_{E_{0}}(1) \otimes \mathcal{O}_{E_{\infty}}(-1)$. 
In this case, for $\omega_{\hat{Y}} \coloneqq \sum_{a \in \hat{\mathcal{A}}}\varepsilon_{a}\omega_{a}$ we have $c_{1}(\hat{\mathcal{L}}) = [\omega_{\hat{Y}}/2\pi]$. 
We say \emph{a blow-down occurs} if $d_{0}>0$ or $d_{\infty}>0$. 
\item Let $\hat{e}_{0}$ and $\hat{e}_{\infty}$ denote the zero and infinity sections of $\hat{X}$, respectively. 
\end{enumerate}

\subsection{Admissible K\"ahler classes and metrics}\label{sec:adm_Kaehler_met}
In this section, we explain a special class of K\"ahler metrics on admissible manifolds, called \emph{admissible K\"ahler metrics}. 
As in the previous section, we refer the reader to \cite{ACGTF08-3} for further details. 

We begin by describing admissible K\"ahler classes, which serve as the natural setting for admissible K\"ahler metrics. 
Let $X = \mathbf{P}_{Y}(E_{0} \oplus E_{\infty})$ be an admissible manifold. 

\begin{definition}
A K\"ahler class $\Omega$ on $X$ is called an \textbf{admissible K\"ahler class} if its pullback to $\hat{X}$ can be written as 
\begin{align}\label{defn_AKC}
c\left(\sum_{a \in \hat{\mathcal{A}}}\frac{[\varepsilon_{a}\omega_{a}]}{x_{a}} + \hat{\Xi}\right)
\end{align}
for some $c \in \mathbf{R}$, $x_{a} \in \mathbf{R} \setminus \{0\}$ with $x_{0} = 1$, $x_{\infty} = -1$, where $\hat{\Xi}$ is the Poincar\'e dual of $2\pi[\hat{e}_{0} + \hat{e}_{\infty}]$. 
If $c = 1$, the admissible K\"ahler class is said to be \emph{normalized}, or is called a \emph{normalized admissible K\"ahler class}. 
\end{definition}

Let $\Xi$ be a 2-form on $X$ whose pullback to $\hat{X}$ is $[\omega_{0}] - [\omega_{\infty}] + \hat{\Xi}$. Then a K\"ahler class $\Omega$ is admissible if and only if 
\[
\Omega = c\left(\sum_{a \in \mathcal{A}}\frac{[\varepsilon_{a}\omega_{a}]}{x_{a}} + \Xi\right)
\]
for some $c \in \mathbf{R}$ and $x_{a} \in \mathbf{R} \setminus \{0\}$. 

\begin{proposition}[{\cite[p.557]{ACGTF08-3} by the discussion therein}]\label{AKC_equiv}
Let $c \in \mathbf{R}$ and $x_{a} \in \mathbf{R} \setminus \{0\}$. 
Let a cohomology class $\Omega \in H^{1, 1}(X, \mathbf{R})$ be given by \eqref{defn_AKC}. 
Then the following two conditions are equivalent: 
\begin{enumerate}[(1)]
\item $\Omega$ is an admissible K\"ahler class of $X$. 
\item $c \in \mathbf{R}_{>0}$ and $0 < \varepsilon_{a}x_{a} < 1$ for any $a \in \mathcal{A}$. 
\end{enumerate}
\end{proposition}

Let $\Omega$ be an admissible K\"ahler class of $X$ given by \eqref{defn_AKC}. 
By Proposition \ref{AKC_equiv}, $c \in \mathbf{R}_{>0}$ and $0 < \varepsilon_{a}x_{a} < 1$ for any $a \in \mathcal{A}$. 
The scalar multiplication on the fiber of $E_{0}$ induces a $\mathbf{C}^{\ast}$-action on $X$. 
By identifying $X_{0} \coloneqq X \setminus (e_{0} \cup e_{\infty}) \cong \hat{X} \setminus (\hat{e}_{0} \cup \hat{e}_{\infty})$ with an open subset of $\hat{X}$, $X_{0}$ is a principal $\mathbf{C}^{\ast}$-bundle over $\hat{Y}$ associated to the holomorphic line bundle $\hat{\mathcal{L}}$. 

We now consider the $S^{1}$-action on $X$ obtained by restricting the $\mathbf{C}^{\ast}$-action. 
The Hermitian metrics on $E_{0}$ and $E_{\infty}$ induce a fiberwise moment map $z \colon X \to [-1, 1]$. 
The critical submanifolds of $z$ are given by 
\[
z^{-1}(\{1\}) = e_{0},\quad z^{-1}(\{-1\}) = e_{\infty}. 
\]
(A concrete construction of $z$ can be found in \cite[Section 2.1]{ACGTF08-3}) 
Let $T$ denote the infinitesimal generator of the $S^{1}$-action on $X$. 
Let $\theta$ be a connection form on $X_{0}$ satisfying 
\begin{align}\label{curvature_of_conn_form}
d\theta = \sum_{a \in \hat{\mathcal{A}}}\varepsilon_{a}\omega_{a}. 
\end{align}

\begin{definition}
A K\"ahler metric $\omega \in \Omega$ on $X$ is called an \emph{admissible K\"ahler metric} if 
\begin{equation}\label{defn_adm_met}
\begin{split}
&g = c\left(\sum_{a \in \hat{\mathcal{A}}}\frac{1+x_{a}z}{x_{a}}\varepsilon_{a}g_{a} + \frac{dz^{2}}{\Theta(z)} + \Theta(z)\theta^{2}\right), \\
&\omega = c\left(\sum_{a \in \hat{\mathcal{A}}}\frac{1+x_{a}z}{x_{a}}\varepsilon_{a}\omega_{a} + dz \wedge \theta \right)
\end{split}
\end{equation}
on $X_{0}$, where $g$ is the Riemannian metric associated to $\omega$, and $\Theta$ is a smooth function on $[-1, 1]$ satisfying 
\begin{enumerate}[(1)]
\item $\Theta(x) > 0$ ($x \in (-1, 1)$), 
\item $\Theta(\pm 1) = 0$, $\Theta'(\pm 1) = \mp 2$. 
\end{enumerate}
If $c = 1$, the admissible K\"ahler metric $\omega$ is said to be \emph{normalized}. 
\end{definition}

Let $\omega \in \Omega$ be an admissible K\"ahler metric in the class $\Omega$ defined by \eqref{defn_adm_met}. 
It is immediate that $z$ is a moment map for $\frac{1}{c}\omega$: 
\[
-dz = \iota_{T}\dfrac{1}{c}\omega. 
\]
Moreover, the complex structure on $X$ corresponding to this K\"ahler structure is obtained by taking the pullback of the complex structures on each $Y_{a}$ ($a \in \mathcal{A}$), together with the condition $d^{c}z = \Theta(z)\theta$. 

\begin{remark}
Let $\omega$ be a K\"ahler metric on $X$, and suppose that it can be written in the form \eqref{defn_adm_met} on $X_{0}$.  
Then, using \eqref{curvature_of_conn_form}, one can show that $\omega$ belongs to $\Omega$. 
Furthermore, the 2-form $d(z\theta)$, regarded as a differential form on $X_{0} \subset \hat{X}$, extends to the whole of $\hat{X}$, and its de Rham cohomology class coincides with $\hat{\Xi}$. 
For details, see \cite[pp. 556--557]{ACGTF08-3}. 
\end{remark}

\subsection{Characteristic polynomials}\label{sec3_characteristic_poly}
In this section, we explain the characteristic polynomial associated with an admissible Kähler class. 
As will be stated in Proposition \ref{prop_chara_poly} below, this polynomial coincides, up to a constant multiple, with the \emph{Duistermaat--Heckman polynomial} associated with the $S^{1}$-action on $X$. 

Let $X = \mathbf{P}_{Y}(E_{0} \oplus E_{\infty})$ be a $d$-dimensional admissible manifold, and let $\Omega$ be an admissible K\"ahler class on $X$ given by \eqref{defn_AKC}. 

\begin{definition}
Let $\lambda_{a} \coloneq \varepsilon_{a}/x_{a}$ for each $a \in \hat{\mathcal{A}}$ (note that $\lambda_{0} = \lambda_{\infty} = 1$), and define 
\begin{align*}
p_{\Omega}(x) 
\coloneqq \prod_{a \in \hat{\mathcal{A}}}(\lambda_{a} + \varepsilon_{a}x)^{d_{a}} 
= (1 + x)^{d_{0}}(1 - x)^{d_{\infty}}\prod_{a \in \mathcal{A}}(\lambda_{a} + \varepsilon_{a}x)^{d_{a}}. 
\end{align*}
We call $p_{\Omega}$ the \emph{characteristic polynomial} of the admissible K\"ahler class $\Omega$. 
\end{definition}

The following proposition can be deduced from a careful reading of the proof of \cite[Proposition 6]{ACGTF08-3}, but we include a short proof for the reader's convenience. 

\begin{proposition}\label{prop_chara_poly}
Let $\omega \in \Omega$ be an admissible K\"ahler metric on $X$ given by \eqref{defn_adm_met}. 
Then 
\begin{align*}
\int_{X}f(z)\frac{\omega^{d}}{d!} = 2\pi c^{d}\left(\prod_{a \in \hat{\mathcal{A}}}\vol(Y_{a}, \omega_{a})\right)\int_{-1}^{1}f(x)p_{\Omega}(x)\,dx
\end{align*}
for any integrable function $f$ on the closed interval $[-1, 1]$. 
\end{proposition}
\begin{proof}
By the well-known multinomial theorem, we have 
\begin{align*}
\frac{\omega^{d}}{d!} 
&= \frac{c^{d}}{d!}\left(\sum_{a \in \hat{\mathcal{A}}}\frac{1 + x_{a}z}{x_{a}}\varepsilon_{a}\omega_{a} + dz \wedge \theta\right)^{d} \\
&= \frac{c^{d}}{(d-1)!}\left(\sum_{a \in \hat{\mathcal{A}}}\frac{1 + x_{a}z}{x_{a}}\varepsilon_{a}\omega_{a} \right)^{d-1} \wedge dz \wedge \theta \\
&= c^{d}p_{\Omega}(z)\left(\bigwedge_{a \in \hat{\mathcal{A}}}\frac{\omega_{a}^{d_{a}}}{d_{a}!}\right) \wedge dz \wedge \theta. 
\end{align*}
Then, integration along the fibers yields
\begin{align*}
\int_{X}f(z)\frac{\omega^{d}}{d!} 
&= c^{d}\int_{X}f(z)p_{\Omega}(z) \left(\bigwedge_{a \in \hat{\mathcal{A}}}\frac{\omega_{a}^{d_{a}}}{d_{a}!}\right) \wedge dz \wedge \theta \\
&= 2\pi c^{d}\left(\prod_{a \in \hat{\mathcal{A}}}\vol(Y_{a}, \omega_{a})\right)\int_{-1}^{1}f(x)p_{\Omega}(x)\,dx, 
\end{align*}
as required. 
\end{proof}

\section{Fano admissible manifolds}\label{sec: Fano_admissible_manifold}
\subsection{Definition of Fano admissible manifolds}
In this section we explain a necessary and sufficient condition for an admissible manifold to be Fano, following \cite{ACGTF08-4}. 
We here assume that $Y = \tilde{Y} = \prod_{a \in \mathcal{A}}Y_{a}$ and each $\omega_{a}$ is a Ricci positive K\"ahler--Einstein metric on $Y_{a}$: $\Ric(\omega_{a}) = \varepsilon_{a}s_{a}\omega_{a}$ ($a \in \mathcal{A}$). 
Let $X$ be an admissible manifold over $Y$. 
Since $Y$ is simply-connected, $X$ can be written as 
\[
X = \mathbf{P}_{Y}(\mathcal{O}^{\oplus(d_{0}+1)} \oplus (\mathcal{L} \otimes \mathcal{O}^{\oplus(d_{\infty}+1)})), 
\]
where $\mathcal{L}$ is a holomorphic line bundle over $Y$ satisfying $\mathcal{L} = \otimes_{a \in \mathcal{A}}\mathcal{L}_{a}$ with each $\mathcal{L}_{a}$ a holomorphic line bundles over $Y_{a}$ satisfying $c_{1}(\mathcal{L}_{a}) = [\varepsilon_{a}\omega_{a}/2\pi]$. 

\begin{theorem}[{\cite[Theorem 3.1]{ACGTF08-4}}]\label{adm_antican_bdle}
For $X$, the following conditions are equivalent: 
\begin{enumerate}[(1)]
\item $X$ is a Fano manifold. 
\item For each $a \in \mathcal{A}$, 
\[
\begin{cases}
s_{a} > d_{0} + 1 & (\text{if }\varepsilon_{a} = 1), \\
s_{a} < -(d_{\infty} + 1) & (\text{if }\varepsilon_{a} = -1). 
\end{cases}
\]
\end{enumerate}
Moreover, if $X$ satisfies these conditions, then $2\pi c_{1}(X)$ is an admissible K\"ahler class, and satisfies 
\begin{align}\label{defn_Fano_AKC}
2\pi c_{1}(X) = c\left[\left(\sum_{a \in \mathcal{A}}\frac{[\varepsilon_{a}\omega_{a}]}{x_{a}}\right) + \Xi\right]
\end{align} 
for 
\begin{align}\label{Fano_cond}
c = \frac{d_{0} + d_{\infty} + 2}{2},\quad x_{a} = \frac{d_{0} + d_{\infty} +2}{2s_{a} + d_{\infty} - d_{0}}\quad (a \in \mathcal{A}). 
\end{align}
\end{theorem}

\begin{definition}
If $X$ satisfies the conditions of Theorem \ref{adm_antican_bdle}, $X$ is called a \emph{Fano admissible manifold}. 
\end{definition}

Here we note that it is shown in \cite{ACGTF08-4} that any such $X$ admits a K\"ahler--Ricci soliton. 

\subsection{Ricci potentials}\label{sec:Ricci_potential}
Let $X = \mathbf{P}_{Y}(\mathcal{O}^{\oplus(d_{0}+1)} \oplus (\mathcal{L} \otimes \mathcal{O}^{\oplus(d_{\infty}+1)}))$ be a Fano admissible manifold. 
Let $\Omega \coloneqq 2\pi c_{1}(X)$, and let $\omega \in \Omega$ be an admissible K\"ahler metric given by \eqref{defn_adm_met}. 
The goal of this section is to compute explicitly the Ricci potential of $\omega$. 
First, by \cite[(3.2)]{ACGTF08-4} we have 
\[
\Ric(\omega) 
= \Ric\left(\frac{\omega}{c}\right) 
= \sum_{a \in \hat{\mathcal{A}}}\Ric(\omega_{a}) - \frac{1}{2}dd^{c}\log(p_{\Omega}(z)\Theta(z))
\]
on $X \setminus (e_{0} \cup e_{\infty})$. 
On the other hand, 
\begin{align*}
\omega 
&= c\left(\sum_{a \in \hat{\mathcal{A}}}\frac{1 + x_{a}z}{x_{a}}\varepsilon_{a}\omega_{a} + dz \wedge \theta\right) \\
&= \sum_{a \in \hat{\mathcal{A}}}\frac{c}{x_{a}}\varepsilon_{a}\omega_{a} + cz\sum_{a \in \hat{\mathcal{A}}}\varepsilon_{a}\omega_{a}+ c(dz \wedge \theta), 
\end{align*}
which can be written as 
\begin{align*}
\omega 
&= \sum_{a \in \hat{\mathcal{A}}}s_{a}\varepsilon_{a}\omega_{a} + \left(\frac{d_{\infty} - d_{0}}{2} + cz\right)\sum_{a \in \hat{\mathcal{A}}}\varepsilon_{a}\omega_{a}+ c(dz \wedge \theta) 
\end{align*}
since for any $a \in \hat{\mathcal{A}}$ 
\[
\frac{c}{x_{a}} 
= s_{a} + \frac{d_{\infty} - d_{0}}{2}. 
\]
Hence we have
\begin{align*}
\Ric(\omega) - \omega 
&= \sum_{a \in \hat{\mathcal{A}}}\Ric(\omega_{a}) - \frac{1}{2}dd^{c}\log(p_{\Omega}(z)\Theta(z)) \\
&\quad -\sum_{a \in \hat{\mathcal{A}}}s_{a}\varepsilon_{a}\omega_{a} - \left(\frac{d_{\infty} - d_{0}}{2} + cz\right)\sum_{a \in \hat{\mathcal{A}}}\varepsilon_{a}\omega_{a} - c(dz \wedge \theta). 
\end{align*}
Using $\Ric(\omega_{a}) = \varepsilon_{a}s_{a}\omega_{a}$, this simplifies to 
\begin{align*}
\Ric(\omega) - \omega 
= - \frac{1}{2}dd^{c}\log(p_{\Omega}(z)\Theta(z)) - \left(\frac{d_{\infty} - d_{0}}{2} + cz\right)\sum_{a \in \hat{\mathcal{A}}}\varepsilon_{a}\omega_{a} - c(dz \wedge \theta). 
\end{align*}
Since $d^{c}z = \Theta(z)\theta$, we have 
\begin{align*}
dz \wedge \theta 
= d(z\theta) - zd\theta 
= d\left(\frac{z}{\Theta(z)}d^{c}z\right) - zd\theta, 
\end{align*}
and using ${\displaystyle d\theta = \sum_{a \in \hat{\mathcal{A}}}\varepsilon_{a}\omega_{a}}$, we obtain 
\begin{align*}
\Ric(\omega) - \omega 
&= - \frac{1}{2}dd^{c}\log(p_{\Omega}(z)\Theta(z)) - \left(\frac{d_{\infty} - d_{0}}{2}\right)d\theta - cd\left(\frac{z}{\Theta(z)}d^{c}z\right). 
\end{align*}
Define functions on $(-1, 1)$ by 
\[
a_{\Theta}(x) \coloneqq \frac{1}{\Theta(x)},\quad b_{\Theta}(x) \coloneqq \frac{x}{\Theta(x)}, 
\]
and let $A_{\Theta}, B_{\Theta}$ be their primitives. 
Then 
\begin{align*}
&d\theta 
= d\left(a_{\Theta}(z)d^{c}z\right) 
= d\left(A_{\Theta}'(z)d^{c}z\right) 
= dd^{c}A_{\Theta}(z), \\
&cd\left(\frac{z}{\Theta(z)}d^{c}z\right) 
= cd\left(b_{\Theta}(z)d^{c}z\right) 
= cd\left(B_{\Theta}'(z)d^{c}z\right) 
= dd^{c}(cB_{\Theta}(z)), 
\end{align*}
and hence 
\begin{align*}
\Ric(\omega) - \omega 
&= \frac{1}{2}dd^{c}\left(-(d_{\infty} - d_{0})A_{\Theta}(z) - 2cB_{\Theta}(z) -\log(p_{\Omega}(z)\Theta(z))\right). 
\end{align*}
\begin{lemma}\label{Ricpot_ext}
Define a smooth function $h \colon (-1, 1) \to \mathbf{R}$ by 
\[
h(x) \coloneqq -(d_{\infty} - d_{0})A_{\Theta}(x) - 2cB_{\Theta}(x) -\log(p_{\Omega}(x)\Theta(x)). 
\]
Then $h$ extends uniquely to a smooth function on the closed interval $[-1, 1]$. 
\end{lemma}
\begin{proof}
It suffices to show that the derivative of $h$ 
\begin{align*}
h'(x) 
&= (-(d_{\infty} - d_{0})A_{\Theta}(x)-2cB_{\Theta}(x) - \log(p_{\Omega}(x)\Theta(x)))' \\
&= \frac{-(d_{\infty} - d_{0}) -2cx - \Theta'(x)}{\Theta(x)} - \frac{p_{\Omega}'(x)}{p_{\Omega}(x)}
\end{align*}
extends smoothly to the closed interval $[-1, 1]$. 
We first consider a neighborhood of $x = 1$. 
Since $\Theta(1) = 0$ and $\Theta'(1) = -2$, we can write 
\[
\Theta(x) = -2(x-1) + \Theta_{1}(x)(x-1)^{2}
\]
for some smooth function $\Theta_{1}$ defined near $x = 1$. 
Therefore, we have 
\begin{align*}
\Theta'(x) = -2 + 2\Theta_{1}(x)(x-1) + \Theta_{1}'(x)(x-1)^{2}, 
\end{align*}
and hence 
\begin{align*}
-(d_{\infty} - d_{0}) -2cx - \Theta'(x) 
&= -2d_{\infty} -(d_{0} + d_{\infty} + 2)(x-1)  \\
&\quad - 2\Theta_{1}(x)(x-1) - \Theta_{1}'(x)(x-1)^{2} 
\end{align*}
and 
\begin{align*}
\frac{-(d_{\infty} - d_{0}) -2cx - \Theta'(x)}{\Theta(x)} 
&= \frac{-2d_{\infty}}{(-2 + \Theta_{1}(x)(x-1))(x-1)} \\
&\quad + \frac{-(d_{0} + d_{\infty} + 2) - 2\Theta_{1}(x) - \Theta_{1}'(x)(x-1)}{-2 + \Theta_{1}(x)(x-1)} \\
&= \frac{d_{\infty}}{x-1} + \Theta_{2}(x), 
\end{align*}
where $\Theta_{2}$ is a smooth function defined near $x = 1$. 
On the other hand, writing $p_{\Omega}(x)$ near $x = 1$ as 
\begin{align*}
p_{\Omega}(x) 
= (1 - x)^{d_{\infty}}p_{1}(x), 
\end{align*}
where $p_{1}$ is a polynomial function satisfying $p_{1}(1) \neq 0$, we have 
\[
p_{\Omega}'(x) = -d_{\infty}(1-x)^{d_{\infty}-1}p_{1}(x) + (1 - x)^{d_{\infty}}p_{1}'(x). 
\]
Thus, 
\begin{align*}
- \frac{p_{\Omega}'(x)}{p_{\Omega}(x)} 
&= \frac{d_{\infty}(1-x)^{d_{\infty}-1}p_{1}(x)}{(1 - x)^{d_{\infty}}p_{1}(x)} -\frac{(1 - x)^{d_{\infty}}p_{1}'(x)}{(1 - x)^{d_{\infty}}p_{1}(x)} \\
&= \frac{d_{\infty}}{(1 - x)} + p_{2}(x), 
\end{align*}
where $p_{2}$ is a smooth function defined near $x = 1$. 
Combining the above expressions, we obtain 
\begin{align*}
h'(x)
&= \frac{-(d_{\infty} - d_{0}) -2cx - \Theta'(x)}{\Theta(x)} - \frac{p_{\Omega}'(x)}{p_{\Omega}(x)} \\
&= \frac{d_{\infty}}{x-1} + \Theta_{2}(x) +\frac{d_{\infty}}{(1 - x)} + p_{2}(x) \\
&= \Theta_{2}(x) + p_{2}(x) 
\end{align*}
which shows that $h'$ extends smoothly near $x = 1$. 
The argument near $x = -1$ is similar. 
Therefore, $h'$ extends smoothly to the closed interval $[-1, 1]$, and so does $h$. 
\end{proof}

We still denote this extension by $h$, and define $h_{\omega} \in C^{\infty}(X, \mathbf{R})$ by 
\[
h_{\omega} \coloneqq h(z). 
\]
Noting that $A_{\Theta}$ and $B_{\Theta}$ are defined only up to additive constants, we obtain the following proposition. 

\begin{proposition}\label{adm_Ric_pot}
By choosing $A_{\Theta}$ and $B_{\Theta}$ so that
\[
\int_{X}(1-e^{h_{\omega}})\frac{\omega^{d}}{d!} = 0, 
\]
the function $h_{\omega} = h(z)$ is the Ricci potential of $\omega$. 
\end{proposition}

\begin{remark}
It is shown that the Ricci potential is a function of $z$ in \cite[Corollary 3.3]{MTF11}. 
Proposition \ref{adm_Ric_pot} provides an explicit expression of this function. 
\end{remark}

By Proposition \ref{adm_Ric_pot}, we see that $1-e^{h_{\omega}}$ is a function of $z$. 
In fact the following stronger statement is also known. 
This fact plays a very important role in Sections \ref{sec:ex_v-sol_Fano_adm} and \ref{prf_Main_thm_1}. 

\begin{proposition}[{\cite[Proposition 6]{ACGTF08-3}, \cite[Theorem 9.1]{Ma21}}]\label{orthogonal _proj}
There exist constants $\alpha, \beta \in \mathbf{R}$ such that 
\begin{align*}
\Pi_{\omega}^{G}(1-e^{h_{\omega}}) = \alpha z + \beta. 
\end{align*}
\end{proposition}

Finally, we show the following proposition, which is crucial for proving Theorem \ref{Main_thm_1} in the introduction. 

\begin{proposition}\label{mommap_normalization}
Let $w \coloneqq \dfrac{d_{0} - d_{\infty}}{d_{0} + d_{\infty} + 2}$. 
Then 
\begin{align*}
\int_{X}(z-w) e^{h_{\omega}}\frac{\omega^{d}}{d!} = 0. 
\end{align*}
\end{proposition}
\begin{proof}
We first apply Proposition \ref{prop_chara_poly} to obtain 
\begin{align*}
\int_{X}ze^{h_{\omega}}\frac{\omega^{d}}{d!} 
&= 2\pi c^{d}\left(\prod_{a \in \hat{\mathcal{A}}}\vol(Y_{a}, \omega_{a})\right)\int_{-1}^{1}e^{-(d_{\infty} - d_{0})A_{\Theta}(x) - 2cB_{\Theta}(x)}\frac{x}{\Theta(x)}\,dx. 
\end{align*}
We next compute the integral on the right-hand side. 
Observe that 
\begin{align*}
\int_{-1}^{1}e^{-(d_{\infty} - d_{0})A_{\Theta}(x) - 2cB_{\Theta}(x)}\frac{x}{\Theta(x)}\,dx 
= \int_{-1}^{1}e^{-(d_{\infty} - d_{0})A_{\Theta}(x)}\left(-\frac{1}{2c}e^{-2cB_{\Theta}}\right)'(x)\,dx. 
\end{align*}
Integrating by parts, we obtain 
\begin{align*}
&\int_{-1}^{1}e^{-(d_{\infty} - d_{0})A_{\Theta}(x) - 2cB_{\Theta}(x)}\frac{x}{\Theta(x)}\,dx \\
&= \left[-\frac{1}{2c}e^{-(d_{\infty} - d_{0})A_{\Theta}(x)-2cB_{\Theta}(x)}\right]_{-1}^{1} + w\int_{-1}^{1}e^{-(d_{\infty} - d_{0})A_{\Theta}(x)-2cB_{\Theta}(x)}\frac{1}{\Theta(x)}\,dx. 
\end{align*}
By Lemma \ref{Ricpot_ext}, we have 
\begin{align*}
\lim_{x \to 1}(-(d_{\infty} - d_{0})A_{\Theta}(x) - 2cB_{\Theta}(x)) &= h(1) + \lim_{x \to 1}\log(p_{\Omega}(x)\Theta(x)) = -\infty, \\
\lim_{x \to -1}(-(d_{\infty} - d_{0})A_{\Theta}(x)-2cB_{\Theta}(x)) &= h(-1) + \lim_{x \to -1}\log(p_{\Omega}(x)\Theta(x)) = -\infty. 
\end{align*}
Hence the boundary term vanishes:  
\begin{align*} 
\left[-\frac{1}{2c}e^{-(d_{\infty} - d_{0})A_{\Theta}(x)-2cB_{\Theta}(x)}\right]_{-1}^{1} = 0. 
\end{align*}
Therefore, we obtain 
\begin{align*}
\int_{-1}^{1}e^{-(d_{\infty} - d_{0})A_{\Theta}(x)-2cB_{\Theta}(x)}\frac{x}{\Theta(x)}\,dx 
= w\int_{-1}^{1}e^{-(d_{\infty} - d_{0})A_{\Theta}(x)-2cB_{\Theta}(x)}\frac{1}{\Theta(x)}\,dx.  
\end{align*}
Substituting this into the expression for the integral over $X$, we have 
\begin{align*}
\int_{X}ze^{h_{\omega}}\frac{\omega^{d}}{d!} 
&= 2\pi c^{d}w\left(\prod_{a \in \hat{\mathcal{A}}}\vol(Y_{a}, \omega_{a})\right)\int_{-1}^{1}e^{-(d_{\infty} - d_{0})A_{\Theta}(x)-2cB_{\Theta}(x)}\frac{1}{\Theta(x)}\,dx \\
&= w\int_{X}e^{h_{\omega}}\frac{\omega^{d}}{d!}. 
\end{align*}
\end{proof}

\section{$v$-solitons on Fano admissible manifolds}\label{sec:v-sol_Fano_adm}
\subsection{$v$-solitons}\label{sec:v-sol}
We now temporarily leave the setting of admissible manifolds and consider a general Fano manifold. 
Let $X$ be a $d$-dimensional Fano manifold, and let $T \subset \Aut^{0}(X)$ be a compact torus. 
Let $\Omega \coloneqq 2\pi c_{1}(X)$, and let $\omega \in \Omega$ be a $T$-invariant K\"ahler metric. 
Denote by $\mu_{\omega} \colon X \to \mathfrak{t}^{\ast}$ the moment map associated with the $T$-action. 
We normalize $\mu_{\omega}$ so that for every $\xi \in \mathfrak{t}$, 
\[
\int_{X}\langle\mu_{\omega}, \xi\rangle e^{h_{\omega}}\frac{\omega^{d}}{d!} = 0, 
\]
where $\langle\cdot, \cdot \rangle \colon \mathfrak{t}^{\ast} \times \mathfrak{t} \to \mathbf{R}$ denotes the natural pairing. 
With this normalization, the image $P \coloneqq \mu_{\omega}(X)$ is independent of the choice of $T$-invariant K\"ahler metric $\omega \in \Omega$. 

\begin{definition}
Let $v \in C^{\infty}(P, \mathbf{R}_{> 0})$. 
A $T$-invariant K\"ahler metric $\omega \in \Omega$ is called a \emph{$v$-soliton} if it satisfies 
\begin{align}\label{defn_v-sol}
\Ric(\omega) - \omega = \frac{1}{2}dd^{c}\log(v(\mu_{\omega})). 
\end{align}
\end{definition}

\begin{example}\label{ex_v-sol_KE}
If $v(x) = 1$ for all $x \in P$, then the equation \eqref{defn_v-sol} becomes 
\[
\Ric(\omega) - \omega = 0, 
\]
so a $v$-soliton reduces to a K\"ahler--Einstein metric. 
\end{example}

\begin{example}
Let $\xi \in \mathfrak{t}$, and set $v(x) = e^{\langle x, \xi \rangle}$ for each $x \in P$. 
Then the equation \eqref{defn_v-sol} becomes 
\[
\Ric(\omega) - \omega = \frac{1}{2}dd^{c}\langle \mu_{\omega}, \xi \rangle, 
\]
and such a metric is known to be a K\"ahler--Ricci soliton. 
\end{example}

\begin{example}\label{ex_v-sol_M-sol}
Mabuchi showed in \cite{Ma21} that Mabuchi solitons are a special case of $v$-solitons as follows. 
Let $X$ be a Fano manifold and assume $M_{X} < 1$. 
Fix a maximal compact subgroup $G \subset \Aut^{0}(X)$. 
For a K\"ahler metric $\omega \in \Omega$, the vector field $\grad_{\omega}\Pi_{\omega}^{G}(1-e^{h_{\omega}})$ is real holomorphic on $X$, and in fact induces an $S^{1}$-action on $X$ (see \cite{FM95} and \cite{Ma21}). 
For a suitable constant $a \in \mathbf{R}$, define $\mu_{\omega} \coloneqq \Pi_{\omega}^{G}(1-e^{h_{\omega}}) -a \colon X \to \mathbf{R}$. 
Then $\mu_{\omega}$ is a normalized moment map for this $S^{1}$-action. 
Define a function $v \colon P \to \mathbf{R}$ by $v(x) \coloneqq 1- a - x$. 
Since $M_{X} < 1$, we have $v > 0$ on $P$. 
A direct computation shows that $\omega$ is a Mabuchi soliton if and only if 
\begin{align*}
\Ric(\omega) - \omega = \frac{1}{2}dd^{c}\log(v(\mu_{\omega})). 
\end{align*}
Thus, Mabuchi solitons are $v$-solitons for this choice of $v$. 
\end{example}

Next, we introduce an obstruction to the existence of $v$-solitons, analogous to the Futaki invariant for K\"ahler--Einstein metrics. 
Let $\omega \in \Omega$ be a $T$-invariant K\"ahler metric, and let $\mu_{\omega} \colon X \to \mathfrak{t}^{\ast}$ be the normalized moment map. 
Let $\mu_{\mathrm{DH}}$ denote the Borel measure on $P$ obtained as the pushforward of the Liouville measure $\omega^{d}/d!$ under $\mu_{\omega}$. 
It is known that $\mu_{\mathrm{DH}}$ is independent of the choice of $\omega \in \Omega$. 
This measure is known as the \emph{Duistermaat--Heckman measure} on $P$ associated with the $T$-action on $X$. 

\begin{definition}
Let $v \in C^{\infty}(P, \mathbf{R}_{>0})$ be a weight function and define a linear map $\Fut_{v} \colon \mathfrak{t} \to \mathbf{R}$ by 
\[
\Fut_{v}(\xi) \coloneqq \int_{P}\langle x, \xi \rangle v(x)\,d\mu_{\mathrm{DH}}. 
\]
This is called the \emph{$v$-Futaki invariant} of $X$. 
\end{definition}

\begin{proposition}\label{v-sol_vanish_vFut}
If $\Omega$ contains a $v$-soliton, then $\Fut_{v} = 0$. 
\end{proposition}

\begin{proof}
Let $\omega \in \Omega$ be a $v$-soliton on $X$. 
Then by \eqref{defn_v-sol} there exists a constant $a \in \mathbf{R}$ such that 
\[
\log(v(\mu_{\omega})) = a + h_{\omega}, 
\]
hence $v(\mu_{\omega}) = e^{a + h_{\omega}}$. 
Therefore, for any $\xi \in \mathfrak{t}$ we have 
\begin{align*}
\Fut_{v}(\xi) 
&= \int_{X}\langle \mu_{\omega}, \xi \rangle v(\mu_{\omega})\frac{\omega^{d}}{d!} 
= \int_{X}\langle \mu_{\omega}, \xi \rangle e^{a + h_{\omega}}\frac{\omega^{d}}{d!} \\
&= e^{a}\int_{X}\langle \mu_{\omega}, \xi \rangle e^{h_{\omega}}\frac{\omega^{d}}{d!} 
= 0 
\end{align*}
by the normalization of $\mu_{\omega}$. 
\end{proof}

\begin{remark}\label{M-sol_vFut_vanish}
It is known that the $v$-Futaki invariant associated with Mabuchi solitons always vanishes. 
In fact, using the notation of Example \ref{ex_v-sol_M-sol}, this follows from 
\begin{align*}
\Fut_{v}(\xi) 
&= \int_{X}\langle \mu_{\omega}, \xi \rangle v(\mu_{\omega})\frac{\omega^{d}}{d!} \\
&= \int_{X}\langle \mu_{\omega}, \xi \rangle \frac{\omega^{d}}{d!} 
- \int_{X}\langle \mu_{\omega}, \xi \rangle \Pi_{\omega}^{G}(1-e^{h_{\omega}})\frac{\omega^{d}}{d!} \\
&= \int_{X}\langle \mu_{\omega}, \xi \rangle \frac{\omega^{d}}{d!} 
- \int_{X}\langle \mu_{\omega}, \xi \rangle (1-e^{h_{\omega}})\frac{\omega^{d}}{d!} \\
&= \int_{X}\langle \mu_{\omega}, \xi \rangle \frac{\omega^{d}}{d!} 
- \int_{X}\langle \mu_{\omega}, \xi \rangle \frac{\omega^{d}}{d!} \\
&= 0. 
\end{align*}
\end{remark}

\subsection{Existence result for Fano admissible manifolds}\label{sec:ex_v-sol_Fano_adm}
In this section, we study the existence problem of $v$-solitons on Fano admissible manifolds. 
Let $X = \mathbf{P}(\mathcal{O}^{\oplus(d_{0}+1)} \oplus (\mathcal{L} \otimes \mathcal{O}^{\oplus(d_{\infty}+1)}))$ be a Fano admissible manifold, and let $\Omega \coloneqq 2\pi c_{1}(X)$ be the admissible K\"ahler class given by \eqref{defn_AKC}. 
Consider the $S^{1}$-action on $X$ induced by scalar multiplication on $\mathcal{O}^{\oplus(d_{0}+1)}$, and let $T$, $z$ and $\theta$ be as defined in Section \ref{sec:adm_Kaehler_met}. 
If $\omega \in \Omega$ is an admissible K\"ahler metric, then by Proposition \ref{mommap_normalization}, the normalized moment map is given by
\[
\mu_{\omega} \coloneqq c(z-w), 
\]
where $w = (d_{0} - d_{\infty})/(d_{0} + d_{\infty} + 2)$. 
Let $P \coloneqq \mu_{\omega}(X)$, and let $v \in C^{\infty}(P, \mathbf{R}_{>0})$. 
Define a function $u \in C^{\infty}([-1, 1], \mathbf{R}_{>0})$ by $u(x) \coloneqq v(c(x-w))$. 
Then $\omega$ is a $v$-soliton if and only if it satisfies 
\[
\Ric(\omega) - \omega = \frac{1}{2}dd^{c}\log(u(z)). 
\]

The goal of this section is to prove the following theorem. 

\begin{theorem}\label{existence_adm_v-sol}
For a Fano admissible manifold $X$, the following three conditions are equivalent: 
\begin{enumerate}[(1)]
\item $X$ admits an admissible $v$-soliton. 
\item The $v$-Futaki invariant of $X$ vanishes. 
\item ${\displaystyle \int_{-1}^{1}(x - w)u(x)p_{\Omega}(x)\,dx = 0}$. 
\end{enumerate}
\end{theorem}

\begin{remark}
The equivalence in Theorem \ref{existence_adm_v-sol} is likely known to experts (see \cite[Theorem 1.2]{NN24} for a related result). 
However, since no explicit proof appears in the literature, we include one here for the reader’s convenience. 
\end{remark}

Before proving Theorem \ref{existence_adm_v-sol}, we present some corollaries. 
From Example \ref{ex_v-sol_KE}, we immediately obtain the following, which was obtained in \cite{ACGTF08-4}. 

\begin{corollary}[{\cite[Corollary 3.1]{ACGTF08-4}}]\label{existence_adm_KE}
For a Fano admissible manifold $X$, the following three conditions are equivalent: 
\begin{enumerate}[(1)]
\item $X$ admits a $G$-invariant K\"ahler--Einstein metric. 
\item The Futaki invariant of $X$ vanishes. 
\item ${\displaystyle \int_{-1}^{1}(x - w)p_{\Omega}(x)\,dx = 0}$. 
\end{enumerate}
\end{corollary}

Noting that the $v$-Futaki invariant corresponding to Mabuchi solitons always vanishes (cf. Proposition \ref{orthogonal _proj} and Remark \ref{M-sol_vFut_vanish}), Theorem \ref{existence_adm_v-sol} yields the following corollary.

\begin{corollary}\label{existence_adm_M-sol}
A Fano admissible manifold $X$ admits a Mabuchi soliton if and only if $M_{X} < 1$. 
\end{corollary}

Now we shall prove Theorem \ref{existence_adm_v-sol}. 
The following proof proceeds largely along the lines of \cite[Theorem 3.1]{ACGTF08-4}. 
The implication (1) $\Rightarrow$ (2) follows from Proposition \ref{v-sol_vanish_vFut}. 
The equivalence of (2) and (3) is clear.
Thus, it remains to prove (3) $\Rightarrow$ (1). 
Let $\omega \in \Omega$ be an admissible K\"ahler metric given by \eqref{defn_adm_met}, and define $F \coloneqq p_{\Omega}\Theta \in C^{\infty}([-1, 1], \mathbf{R})$. 
Then $F$ satisfies 
\begin{align}
&F(x) > 0\quad (x \in (-1, 1)), \label{extremal_fct_1} \\
&F(\pm 1) = 0,\quad F'(\pm 1) = \mp 2p_{\Omega}(\pm 1). \label{extremal_fct_2}
\end{align}

\begin{proposition}\label{existence_v-sol_base_KE}
The following conditions are equivalent: 
\begin{enumerate}[(1)]
\item $\omega$ is a $v$-soliton. 
\item $F$ satisfies 
\begin{align}\label{v-sol_eq_F}
\frac{(uF)'(z)}{(up_{\Omega})(z)} = -(d_{0} + d_{\infty} + 2)(z - w). 
\end{align}
\end{enumerate}
\end{proposition}
\begin{proof}
From the computation in Section \ref{sec:Ricci_potential}, on $X \setminus (e_{0} \cup e_{\infty})$, 
\begin{align*}
\Ric(\omega) - \omega 
= \sum_{a \in \hat{\mathcal{A}}}\Ric(\omega_{a})-\sum_{a \in \hat{\mathcal{A}}}c\left(\frac{1}{x_{a}} + z\right)\varepsilon_{a}\omega_{a} - \frac{1}{2}dd^{c}\log(F(z)) - c(dz \wedge \theta). 
\end{align*}
Thus, $\omega$ is a $v$-soliton if and only if 
\begin{equation}\label{v-sol_eq_F_0}
\sum_{a \in \hat{\mathcal{A}}}\Ric(\omega_{a})-\sum_{a \in \hat{\mathcal{A}}}c\left(\frac{1}{x_{a}} + z\right)\varepsilon_{a}\omega_{a} - \frac{1}{2}dd^{c}\log(u(z)F(z)) - c(dz \wedge \theta) = 0. 
\end{equation}
Computing $dd^{c}\log(u(z)F(z))$, we obtain 
\begin{align*}
dd^{c}\log(u(z)F(z)) 
&= d\left(\frac{(uF)'(z)}{(uF)(z)}d^{c}z\right) 
= d\left(\frac{(uF)'(z)}{(up_{\Omega})(z)}\theta\right) \\
&= \left(\frac{(uF)'}{up_{\Omega}}\right)'(z)dz \wedge \theta + \frac{(uF)'(z)}{(up_{\Omega})(z)}d\theta \\
&= \left(\frac{(uF)'}{up_{\Omega}}\right)'(z)dz \wedge \theta + \frac{(uF)'(z)}{(up_{\Omega})(z)}\sum_{a \in \hat{\mathcal{A}}}\varepsilon_{a}\omega_{a}. 
\end{align*}
Thus, the equation \eqref{v-sol_eq_F_0} is equivalent to 
\begin{align*}
&\sum_{a \in \hat{\mathcal{A}}}\left[\Ric(\omega_{a}) - \left\{\frac{1}{2}\frac{(uF)'(z)}{(up_{\Omega})(z)} + c\left(\frac{1}{x_{a}} + z\right) \right\}\varepsilon_{a}\omega_{a}\right] \\
&\quad - \frac{1}{2}\left\{\left(\frac{(uF)'}{up_{\Omega}}\right)'(z) +2c\right\}dz \wedge \theta = 0, 
\end{align*}
in other words, 
\begin{align}
&\Ric(\omega_{a}) - \left\{\frac{1}{2}\frac{(uF)'(z)}{(up_{\Omega})(z)} + c\left(\frac{1}{x_{a}} + z\right) \right\}\varepsilon_{a}\omega_{a} = 0 \quad (a \in \hat{\mathcal{A}}), \label{v-sol_eq_F_1} \\
&\left(\frac{(uF)'}{up_{\Omega}}\right)'(z) +2c = 0. \label{v-sol_eq_F_2}
\end{align}
Furthermore, since each $\omega_{a}$ is a K\"ahler--Einstein metric the equation \eqref{v-sol_eq_F_1} reduces to \begin{align}\label{v-sol_eq_F_3}
&\frac{1}{2}\frac{(uF)'(z)}{(up_{\Omega})(z)} + c\left(\frac{1}{x_{a}} + z\right) = s_{a}\quad (a \in \hat{\mathcal{A}}). 
\end{align}
The equation \eqref{v-sol_eq_F_2} is obtained by the derivative of \eqref{v-sol_eq_F_3}. 
Thus, $\omega$ is a $v$-soliton if and only if $F$ satisfies \eqref{v-sol_eq_F_3}. 
As in \eqref{Fano_cond}, we have 
\begin{align*}
s_{a} - c\left(\frac{1}{x_{a}} + z\right) 
&= \frac{1}{2}[(d_{0} + 1)(1-z) - (d_{\infty} + 1)(1+z)] \\
&= -\frac{1}{2}(d_{0} + d_{\infty} + 2)(z - w),
\end{align*} 
and hence \eqref{v-sol_eq_F_3} reduces to a single equation \eqref{v-sol_eq_F}. 
\end{proof}

We are now ready to prove the implication (3) $\Rightarrow$ (1). 
Assume $X$ satisfies (3), and consider the boundary value problem 
\begin{align}
&\frac{(uF)'(x)}{(up_{\Omega})(x)} = -(d_{0} + d_{\infty} + 2)(x - w), \label{v-sol_eq_F_4} \\
&F(\pm 1) = 0. \label{v-sol_eq_F_5}
\end{align}
The solution of \eqref{v-sol_eq_F_4} satisfying $F(-1) = 0$ is given by
\begin{align*}
F(x) 
= -\frac{d_{0} + d_{\infty} + 2}{u(x)}\int_{-1}^{x}(t - w)u(t)p_{\Omega}(t)\,dt. 
\end{align*}
This function $F$ is smooth on the closed interval $[-1, 1]$. 
Moreover, by assumption (3), we also have $F(1) = 0$. 

Next, we compute the derivatives at the endpoints. 
We obtain 
\begin{align*}
F'(\pm 1) 
&= -(d_{0} + d_{\infty} + 2)(\pm 1 - w)p_{\Omega}(\pm 1) \\
&= 
\begin{cases}
2(d_{0} + 1)p_{\Omega}(-1) & \text{if $x = -1$}, \\
-2(d_{\infty} + 1)p_{\Omega}(1) & \text{if $x = 1$}. 
\end{cases}
\end{align*}
In particular, 
if $d_{0} > 0$, then $p_{\Omega}(-1) = 0$ and hence $F'(-1) = 2p_{\Omega}(-1) = 0$. 
If $d_{0} = 0$, then clearly $F'(-1) = 2p_{\Omega}(-1)$. 
Thus, in all cases, we have $F'(-1) = 2p_{\Omega}(-1)$. 
A similar argument shows that $F'(1) = -2p_{\Omega}(1)$. 
Therefore, $F$ satisfies the condition \eqref{extremal_fct_2}. 

We now verify that $F$ satisfies the condition \eqref{extremal_fct_1}. 
It suffices to show that $uF > 0$ on the open interval $(-1, 1)$. 
First note that $w \in (-1, 1)$ because 
\begin{align*}
w 
= \frac{d_{0} - d_{\infty}}{d_{0} + d_{\infty} + 2} 
\geq -\frac{d_{\infty}}{d_{0} + d_{\infty} + 2} 
> -1, 
\end{align*}
and 
\begin{align*}
w 
= \frac{d_{0} - d_{\infty}}{d_{0} + d_{\infty} + 2} 
\leq \frac{d_{0}}{d_{0} + d_{\infty} + 2} 
< 1. 
\end{align*}
Let $x \in (-1, 1)$ arbitrarily. 
If $x \leq w$, then $t - w < 0$ on $[-1, x)$, and hence 
\begin{align*}
(uF)(x) 
= -(d_{0} + d_{\infty} + 2)\int_{-1}^{x}(t - w)u(t)p_{\Omega}(t)\,dt 
> 0. 
\end{align*}
If $x > w$, then $t - w > 0$ on $[x, 1)$, and we obtain 
\begin{align*}
(uF)(x) 
&= -(d_{0} + d_{\infty} + 2)\left[\int_{-1}^{1}(t - w)u(t)p_{\Omega}(t)\,dt - \int_{x}^{1}(t - w)u(t)p_{\Omega}(t)\,dt\right] \\
&= (d_{0} + d_{\infty} + 2)\int_{x}^{1}(t - w)u(t)p_{\Omega}(t)\,dt \\
&> 0. 
\end{align*}
Thus, $uF > 0$ holds for all $x \in (-1, 1)$. 

Finally, define $\Theta \colon (-1, 1) \to \mathbf{R}$ by $\Theta \coloneqq F/p_{\Omega}$. 
It is clear that $\Theta$ is a positive smooth function on the open interval $(-1, 1)$, and from the expression  
\begin{align*}
\Theta(x) 
&= -\frac{d_{0} + d_{\infty} + 2}{u(x)p_{\Omega}(x)}\int_{-1}^{x}(t - w)u(t)p_{\Omega}(t)\,dt, 
\end{align*}
we analyze its behavior near $x = -1$. 
Since $u(x) \neq 0$ and $x-w \neq 0$ near $x = -1$, we can write 
\begin{align*}
&u(x)p_{\Omega}(x) = (1 + x)^{d_{0}}g_{-}(x),\quad 
\int_{-1}^{x}(t - w)u(t)p_{\Omega}(t)\,dt = (1 + x)^{d_{0}+1}h_{-}(x), 
\end{align*}
where $g_{-}$, $h_{-}$ are smooth around $x = -1$ and nonvanishing at $x = -1$. 
Then, we have
\begin{align*}
(1 + x)^{d_{0}}(x - w)g_{-}(x) &= (x - w)u(x)p_{\Omega}(x) \\
&= \left(\int_{-1}^{x}(t - w)u(t)p_{\Omega}(t)\,dt\right)'  \\
&= (d_{0}+1)(1 + x)^{d_{0}}h_{-}(x) + (1 + x)^{d_{0}+1}h_{-}'(x) \\
&= (1 + x)^{d_{0}}\left[(d_{0}+1)h_{-}(x) + (1 + x)h_{-}'(x) \right]
\end{align*}
and 
\begin{align*}
(x - w)g_{-}(x) 
= (d_{0}+1)h_{-}(x) + (1 + x)h_{-}'(x). 
\end{align*}
This shows that 
\begin{align*}
\Theta(x) 
&= -\frac{(d_{0} + d_{\infty} + 2)(1 + x)^{d_{0}+1}h_{-}(x)}{(1 + x)^{d_{0}}g_{-}(x)} \\
&= -(1 + x)\frac{(d_{0} + d_{\infty} + 2)h_{-}(x)}{g_{-}(x)}, 
\end{align*}
and hence $\Theta$ extends smoothly to $x = -1$ with 
\begin{align*}
\Theta(-1) = 0,\quad \Theta'(-1) = 2. 
\end{align*}
A similar argument shows that $\Theta$ extends smoothly to $x = 1$ with 
\begin{align*}
\Theta(1) = 0,\quad \Theta'(1) = -2. 
\end{align*}
Therefore, $\Theta$ satisfies all the required conditions, and hence defines an admissible K\"ahler metric $\omega \in \Omega$. 
By construction, this metric is a $v$-soliton. 
This completes the proof of Theorem \ref{existence_adm_v-sol}. 

\section{The Mabuchi constant of Fano admissible manifolds}\label{prf_Main_thm_1}
The purpose of this section is to prove Theorem \ref{Main_thm_1}. 
Let $X = \mathbf{P}(\mathcal{O}^{\oplus(d_{0}+1)} \oplus (\mathcal{L} \otimes \mathcal{O}^{\oplus(d_{\infty}+1)}))$ be a Fano admissible manifold, and let $\Omega \coloneqq 2\pi c_{1}(X)$. 
For $i \in \{0, 1, 2\}$, set 
\begin{align*}
b_{i} \coloneqq \int_{-1}^{1}x^{i}p_{\Omega}(x)\,dx. 
\end{align*}
Since $p_{\Omega} > 0$ on $(-1,1)$, we have $b_{0}, b_{2} > 0$. 
Moreover, by the Cauchy--Schwarz inequality, $b_{0}b_{2}-b_{1}^{2} > 0$. 

Using Propositions \ref{orthogonal _proj} and \ref{mommap_normalization} in Section 4.2, we obtain the following. 

\begin{theorem}[{\cite[Theorem 6.5]{Ma01}} for cases $d_{0} = d_{\infty} = 0$]\label{M-const_proj_bdle}
Define $w \coloneqq \frac{d_{0} - d_{\infty}}{d_{0} + d_{\infty} + 2}$. 
Then 
\begin{align*}
M_{X} 
= \frac{b_{0}|b_{1} - wb_{0}| - b_{1}(b_{1} - wb_{0})}{b_{0}b_{2} - b_{1}^{2}} 
= 1 + \frac{b_{0}(|b_{1} - wb_{0}| - (b_{2} - wb_{1}))}{b_{0}b_{2}-b_{1}^{2}}. 
\end{align*}
\end{theorem}
\begin{proof}
Let $\omega \in \Omega$ be an admissible K\"ahler metric given by \eqref{defn_adm_met}. 
Since constant functions and $z$ are Killing potentials of $(X, \omega)$, Propositions \ref{prop_chara_poly} and \ref{orthogonal _proj} imply that 
\begin{align*}
0 
&= \int_{X}(1-e^{h_{\omega}})\frac{\omega^{d}}{d!} 
= \int_{X}\Pi_{\omega}^{G}(1-e^{h_{\omega}})\frac{\omega^{d}}{d!} \\
&= \int_{X}(\alpha z + \beta)\frac{\omega^{d}}{d!} 
= 2\pi c^{d}\left(\prod_{a \in \hat{\mathcal{A}}}\vol(Y_{a}, \omega_{a})\right)(\alpha b_{1} + \beta b_{0}). 
\end{align*}
Hence we obtain 
\begin{align}\label{Comp_M_const_1}
\alpha b_{1} + \beta b_{0} = 0. 
\end{align}
Furthermore, 
\begin{align*}
\int_{X}z(1-e^{h_{\omega}})\frac{\omega^{d}}{d!} 
&= \int_{X}z\Pi_{\omega}^{G}(1-e^{h_{\omega}})\frac{\omega^{d}}{d!} 
= \int_{X}(\alpha z^{2} + \beta z)\frac{\omega^{d}}{d!} \\
&= 2\pi c^{d}\left(\prod_{a \in \hat{\mathcal{A}}}\vol(Y_{a}, \omega_{a})\right)(\alpha b_{2} + \beta b_{1}). 
\end{align*}
On the other hand, by Proposition \ref{mommap_normalization}, we have 
\begin{align*}
\int_{X}z(1-e^{h_{\omega}})\frac{\omega^{d}}{d!} 
&= \int_{X}z\frac{\omega^{d}}{d!}  - w\int_{X}e^{h_{\omega}}\frac{\omega^{d}}{d!} 
= \int_{X}z\frac{\omega^{d}}{d!}  - w\int_{X}\frac{\omega^{d}}{d!} \\
&= 2\pi c^{d}\left(\prod_{a \in \hat{\mathcal{A}}}\vol(Y_{a}, \omega_{a})\right)(b_{1} -wb_{0}). 
\end{align*}
These give us that 
\begin{align}\label{Comp_M_const_2}
\alpha b_{2} + \beta b_{1} 
= b_{1} -wb_{0}. 
\end{align}
By solving \eqref{Comp_M_const_1} and \eqref{Comp_M_const_2}, we obtain 
\begin{align*}
\alpha = \frac{b_{0}b_{1}-wb_{0}^{2}}{b_{0}b_{2}-b_{1}^{2}},\quad 
\beta = -\frac{b_{1}^{2} - wb_{0}b_{1}}{b_{0}b_{2}-b_{1}^{2}}, 
\end{align*}
and 
\begin{align*}
M_{X} 
&= \max_{x \in [-1,1 ]}(\alpha z + \beta) 
= |\alpha| + \beta \\
&= \frac{b_{0}|b_{1} - wb_{0}| - b_{1}(b_{1} - wb_{0})}{b_{0}b_{2} - b_{1}^{2}} \\
&= 1 + \frac{b_{0}(|b_{1} - wb_{0}| - (b_{2} - wb_{1}))}{b_{0}b_{2}-b_{1}^{2}}, 
\end{align*}
as required. 
\end{proof}

\section{Mabuchi solitons on Fano admissible manifolds over $\mathbf{P}^{n}$}\label{sec_computation}
In the final section, we concern the existence and non-existence of Mabuchi solitons on Fano admissible manifolds over $\mathbf{P}^{n}$. 
Let $X$ be a Fano admissible manifold over $\mathbf{P}^{n}$. 
Then $X$ can be written as 
\begin{align}\label{Fano_adm_over_P^n}
X = \mathbf{P}_{\mathbf{P}^n}(\mathcal{O}^{\oplus(d_{0}+1)}\oplus\mathcal{O}(k)^{\oplus(d_{\infty}+1)}) 
\end{align}
for some $k \in \mathbf{Z} \setminus \{0\}$ and $d_{0}, d_{\infty} \in \mathbf{Z}_{\geq 0}$. 
In that case, 
\[
\mathbf{P}_{\mathbf{P}^n}(\mathcal{O}^{\oplus(d_{0}+1)}\oplus\mathcal{O}(k)^{\oplus(d_{\infty}+1)}) \cong \mathbf{P}_{\mathbf{P}^n}(\mathcal{O}^{\oplus(d_{\infty}+1)}\oplus\mathcal{O}(-k)^{\oplus(d_{0}+1)})
\]
and so we may assume $k \in \mathbf{Z}_{\geq 1}$. 
By Theorem \ref{adm_antican_bdle}, $X$ is Fano if and only if $k(d_{0}+1)<n+1$. 
In what follows, we assume these conditions. 
Moreover, by Corollary \ref{existence_adm_M-sol}, the existence of a Mabuchi soliton on $X$ is equivalent to $M_{X}< 1$. 
The purpose of this section is to prove the following theorem. 

\begin{theorem}\label{M-sol_adm/P^n} 
The Fano admissible manifold $X$ admits a Mabuchi soliton if and only if 
\[
(k, d_{\infty})=(1,0)\quad \text{or}\quad (n, k, d_{0}, d_{\infty}) = (1, 1, 0, 1). 
\]
In all other cases, $M_{X} > 1$. 
\end{theorem}

To prove Theorem \ref{M-sol_adm/P^n}, we compute $M_{X}$ by using Theorem \ref{M-const_proj_bdle}. 
Let
\[
\lambda \coloneq \frac{2(n+1)+k(d_{\infty}-d_{0})}{k(d_{0}+d_{\infty}+2)}. 
\]
Then, for each $i \in \{0,1,2\}$, 
\[
b_{i}=\int_{-1}^{1}(\lambda + x)^{n}(1+x)^{d_{0}}(1-x)^{d_{\infty}}x^{i}dx. 
\]

\begin{proposition}\label{nonvanishing_Futaki}
$b_{1}-wb_{0} > 0$. 
In particular, we have 
\[
M_X=1+\frac{b_0(b_1-wb_0+wb_1-b_2)}{b_0b_2-b_1^2}. 
\]
\end{proposition}

\begin{proof}
Define $a \in \mathbf{R}$ by 
\[
a \coloneqq \dfrac{n+1-k(d_0+1)}{k(d_0+d_\infty+2)}. 
\]
By assumption, $a > 0$ and 
  \[
  \lambda-1
=\frac{2[n+1-k(d_0+1)]}{k(d_0+d_\infty+2)}
=2a. 
  \]
Therefore, 
\begin{align*}
b_{1} - wb_{0} 
&=\int_{-1}^1(\lambda+x)^n(1+x)^{d_0}(1-x)^{d_\infty}(x-w)\,dx \\
&=\int_{0}^{1}(\lambda+2u-1)^n(2u)^{d_0}(2-2u)^{d_\infty}(2u-1-w)\,du \\
&= 2^{d_{0} + d_{\infty} + n + 2}\int_{0}^{1}(a + u)^n u^{d_0}(1 - u)^{d_\infty}\left(u - \frac{d_{0} + 1}{d_{0} + d_{\infty} + 2}\right)\,du. 
\end{align*}
On the other hand, 
\begin{align*}
&(d_{0} + d_{\infty} + 2) u^{d_0}(1 - u)^{d_\infty}\left(u - \frac{d_{0} + 1}{d_{0} + d_{\infty} + 2}\right) \\
&= (d_{0} + d_{\infty} + 2) u^{d_0 + 1}(1 - u)^{d_\infty} - (d_{0} + 1)u^{d_0}(1 - u)^{d_\infty} \\
&= u^{d_0}(1 - u)^{d_\infty}[(d_{0} + d_{\infty} + 2)u - (d_{0} + 1)] \\&= u^{d_0}(1 - u)^{d_\infty}[(d_{0} + 1)(u-1) + (d_{\infty} + 1)u] \\
&= -(u^{d_0 + 1}(1 - u)^{d_\infty + 1})'. 
\end{align*}
Hence, integration by parts yields 
\begin{align*}
b_{1} - wb_{0} 
&=-\frac{2^{d_{0} + d_{\infty} + n + 2}}{d_{0} + d_{\infty} + 2}\int_{0}^{1}(a + u)^n (u^{d_0 + 1}(1 - u)^{d_\infty + 1})'\,du \\
&=\frac{n2^{d_{0} + d_{\infty} + n + 2}}{d_{0} + d_{\infty} + 2}\int_{0}^{1}(a + u)^{n-1} u^{d_0 + 1}(1 - u)^{d_\infty + 1}\,du \\
&= \frac{n2^{d_{0} + d_{\infty} + n + 2}}{d_{0} + d_{\infty} + 2}\sum_{j=0}^{n-1}\binom{n-1}{j}a^{n-1-j}B(j+d_{0}+2, d_{\infty}+2), 
\end{align*}
where $B(x, y)$ is the beta function. 
Therefore, $b_{1} - wb_{0} > 0$ follows from $a > 0$. 
The second assertion follows immediately from Theorem \ref{M-const_proj_bdle}. 
\end{proof}

Using Corollary \ref{existence_adm_KE}, we obtain the following corollary. 

\begin{corollary}\label{nonexistence_KE}
For any $k \in \mathbf{Z}_{\ge 1}$ and $d_{0}, d_{\infty} \in \mathbf{Z}_{\geq 0}$, $X = \mathbf{P}_{\mathbf{P}^n}(\mathcal{O}^{\oplus(d_{0}+1)}\oplus\mathcal{O}(k)^{\oplus(d_{\infty}+1)})$ does not admit a K\"ahler--Einstein metric. 
\end{corollary}

Now, set 
\[
I\coloneq b_1-wb_0+wb_1-b_2. 
\]
Then 
\begin{align*}
I
  &=\int_{-1}^1(\lambda+x)^n(1+x)^{d_0}(1-x)^{d_\infty+1}(x-w)dx
\end{align*}
holds. 
Moreover, the condition $M_{X} \leq 1$ (respectively $M_{X} = 1$) is equivalent to $I  \leq 0$ (respectively $I = 0$). 

\begin{lemma}\label{inequity}
If the following inequality holds, then $M_{X} > 1$ holds: 
\begin{equation}\label{eq1}
\frac{n(d_\infty+1)-(d_0+1)}{d_0+d_\infty+4}\ge\frac{n+1-k(d_0+1)}{k(d_0+d_\infty+2)}.     
\end{equation}
In particular, in that case $X$ does not admit a Mabuchi soliton. 
\end{lemma}

\begin{proof}
Define $b \in \mathbf{R}$ by 
\[
b \coloneqq \frac{d_0+1}{d_0+d_\infty+2}. 
\]
Then $b > 0$, and 
  \[
  w+1
  =\frac{2(d_0+1)}{d_0+d_\infty+2}=2b. 
  \]
Moreover, 
\begin{equation}\label{I_expression}
\begin{split}
  I
  &=\int_0^1(\lambda+2u-1)^{n}(2u)^{d_0}(2-2u)^{d_\infty+1}(2u-(w+1))2\,du\\
  &=2^{d_0+d_\infty+n+3}\int_0^1(a+u)^nu^{d_0}(1-u)^{d_\infty+1}(u-b)\,du
\end{split}
\end{equation}
holds, where $a$ is defined in the proof of Proposition \ref{nonvanishing_Futaki}. 
Since 
\begin{align*}
  &\int_0^1(a+u)^nu^{d_0}(1-u)^{d_\infty+1}(u-b)du\\
  &=\sum_{j=0}^n\binom{n}{j}a^{n-j}\left[\int_0^1u^{j+d_0+1}(1-u)^{d_\infty + 1}du-b\int_0^1 u^{j+d_0}(1-u)^{d_\infty + 1}du\right]\\
  &=\sum_{j=0}^n\binom{n}{j}a^{n-j}[B(j+d_0+2,d_\infty+2)-bB(j+d_0+1,d_\infty+2)]
\end{align*}
and 
\begin{align*}
  &B(j+d_0+2,d_\infty+2)-bB(j+d_0+1,d_\infty+2) \\
  &=\left[\frac{j+d_0+1}{(j+d_0+1)+(d_\infty+2)}-\frac{d_0+1}{d_0+d_\infty+2}\right]B(j+d_0+1,d_\infty+2)\\
  &=\frac{(d_\infty+1)j-(d_0+1)}{d_0+d_\infty+2}\frac{B(j+d_0+1,d_\infty+2)}{j+d_0+d_\infty+3}\\
  &=\frac{(d_\infty+1)j-(d_0+1)}{d_0+d_\infty+2}\frac{B(j+d_0+1,d_\infty+3)}{d_\infty+2}, 
\end{align*}
we have 
 \[
 I=\frac{2^{d_0+d_\infty+n+3}}{(d_0+2)(d_0+d_\infty+2)}\sum_{j=0}^n\binom{n}{j}a^{n-j}[(d_\infty+1)j-(d_0+1)]B(j+d_0+1,d_\infty+3). 
 \]
Therefore, if we set 
  \[
  S_0\coloneq\sum_{j=0}^n\binom{n}{j}a^{n-j}B(d_0+1+j,d_\infty+3),\quad S_1\coloneq\sum_{j=0}^nj\binom{n}{j}a^{n-j}B(d_0+1+j,d_\infty+3), 
  \]
then, since 
  \[
 I=\frac{2^{d_0+d_\infty+n+3}}{(d_0+2)(d_0+d_\infty+2)}[(d_\infty+1)S_1-(d_0+1)S_0]
 \] 
and $S_0>0$, the condition $I\ge0$ (respectively, $I = 0$) is equivalent to 
 \begin{align}\label{equiv_I_nneg_S_0_S_1}
 \frac{S_1}{S_0}\ge\frac{d_0+1}{d_\infty+1}\quad \left(\text{respectively, } \frac{S_1}{S_0} = \frac{d_0+1}{d_\infty+1}
\right). 
 \end{align}
Since 
\begin{align*}
  S_0
  &=\sum_{j=0}^n\binom{n}{j}a^{n-j}\int_0^1t^{j+d_0}(1-t)^{d_\infty+2}dt\\
  &=\int_0^1t^{d_0}(1-t)^{d_\infty+2}(a+t)^ndt\\
  &=a\int_0^1t^{d_0}(1-t)^{d_\infty+2}(a+t)^{n-1}dt+\int_0^1t^{d_0+1}(1-t)^{d_\infty+2}(a+t)^{n-1}dt, 
\end{align*}
using 
\begin{align*}
  nx(x+y)^{n-1}
  &=x\frac{\partial}{\partial x}(x+y)^n
  =\sum_{j=0}^nj\binom{n}{j}x^jy^{n-j}, 
\end{align*}
we obtain 
 \begin{align*}
  S_1
  &=\sum_{j=0}^nj\binom{n}{j}a^{n-j}\int_0^1t^{d_0+j}(1-t)^{d_\infty+2}dt\\
  &=\int_0^1t^{d_0}(1-t)^{d_\infty+2}nt(a+t)^{n-1}dt\\
  &=n\int_0^1t^{d_0+1}(1-t)^{d_\infty+2}(a+t)^{n-1}dt. 
\end{align*}
Now let us set 
 \[
 K\coloneq\int_0^1t^{d_0+1}(1-t)^{d_\infty+2}(a+t)^{n-1}dt,\quad L\coloneq\int_{0}^1t^{d_0}(1-t)^{d_\infty+2}(a+t)^{n-1}dt. 
 \]
Then $K, L > 0$, and 
 \[
 \frac{S_1}{S_0}=\frac{nK}{aL+K}. 
 \]
Define $f\colon(0,+\infty)\to\mathbf{R}$ by 
\[
f(t)\coloneq\frac{nt}{aL+t}. 
\]
Then, since $aL >0$, $f$ is a strictly increasing function. 
Next, if we set 
 \[
 g(t)\coloneq t^{d_0}(1-t)^{d_\infty+2},\quad h(t)\coloneq(a+t)^{n-1}, 
 \]
then 
 \begin{align*}
  &K=\int_0^1tg(t)h(t)\,dt,\quad L=\int_0^1g(t)h(t)\,dt, \\
  &\int_0^1g(t)\,dt=B(d_0+1,d_\infty+3),\quad \int_0^1tg(t)\,dt=B(d_0+2,d_\infty+3)
 \end{align*}
hold. 
Here, we investigate the sign of 
\[
B(d_0+1,d_\infty+3)K-B(d_0+2,d_\infty+3)L. 
\]
 \begin{itemize}
   \item In the case $n=1$. \par
     Since $h(t)=1$, we have $K=B(d_0+2,d_\infty+3), L=B(d_0+1,d_\infty+3)$ and hence $B(d_0+1,d_\infty+3)K-B(d_0+2,d_\infty+3)L=0$. 
     \item In the case $n\ge2$. \par
     Since $h'(t)=(n-1)(a+t)^{n-2} > 0 $, $h$ is strictly increasing on $[0,1]$, and 
 \begin{align*}
   &B(d_0+1,d_\infty+3)K-B(d_0+2,d_\infty+3)L\\
   &=\int_0^1tg(t)h(t)\,dt\int_0^1g(t)\,dt-\int_0^1g(t)h(t)dt\int_0^1tg(t)\,dt\\
   &=\int_0^1tg(t)h(t)\,dt\int_0^1g(s)\,ds-\int_0^1g(t)h(t)\,dt\int_0^1sg(s)\,ds\\
   &=\int_0^1\int_0^1(t-s)h(t)g(s)g(t)\,dsdt
 \end{align*}
 holds. 
 By a change of variables, we have 
 \[
   \int_0^1\int_0^1(t-s)h(t)g(s)g(t)\,dsdt=\int_0^1\int_0^1(s-t)h(s)g(t)g(s)\,dsdt
 \]
 and 
 \begin{align*}
   &B(d_0+1,d_\infty+3)K-B(d_0+2,d_\infty+3)L\\
   &=\frac{1}{2}\int_0^1\int_0^1((t-s)h(t)+(s-t)h(s))g(s)g(t)\,dsdt\\
   &=\frac{1}{2}\int_0^1\int_0^1(h(s)-h(t))(s-t)g(s)g(t)\,dsdt. 
 \end{align*} 
Therefore, since $(h(s)-h(t))(s-t)\ge0$, we have $B(d_0+1,d_\infty+3)K - B(d_0+2,d_\infty+3)L \ge0$. 
 \end{itemize}
From the above, 
\[
K \ge \frac{B(d_0+2,d_\infty+3)}{B(d_0+1,d_\infty+3)} L = \frac{d_0+1}{d_0+d_\infty+4}L
 \]
holds, and so, by the monotonicity of $f$, 
 \[
 \frac{S_1}{S_0}=\frac{nK}{aL+K}\ge\frac{n(d_0+1)}{a(d_0+d_\infty+4) + d_0+1}
 \]
follows. 
Therefore, if \eqref{eq1} holds, then 
 \begin{align*}
 \frac{n(d_0+1)}{a(d_0+d_\infty+4)+(d_0+1)}\ge\frac{d_0+1}{d_\infty+1}, 
 \end{align*}
and so \eqref{equiv_I_nneg_S_0_S_1} follows. 
In particular, since $I\ge0$, we have $M_{X} \geq 1$. 
Furthermore, the following holds. 
\begin{enumerate}[(1)]
  \item In the case $n=1$. \par
    Since $k(d_0+1)<n+1$, we have $n=k=1$ and $d_0=0$.
 Hence, it holds 
    \[
    a=\frac{n+1 - k(d_0+1)}{k(d_0+d_\infty+2)}=\frac{1}{d_\infty+2} 
    \]
    and $B(d_0+1,d_\infty+3)K-B(d_0+2,d_\infty+3)L=0$. 
    In particular, we have 
    \[
    K = \frac{d_{0}+1}{d_{0}+d_{\infty}+4}L. 
    \]
    Here, if we assume $M_{X} = 1$, then $I = 0$, and so by \eqref{equiv_I_nneg_S_0_S_1} we obtain 
    \begin{align*}
    \frac{d_0+1}{d_\infty+1}=\frac{S_1}{S_0}=\frac{nK}{aL+K}=\frac{n(d_0+1)}{a(d_0+d_\infty+4)+(d_0+1)}. 
    \end{align*}
    Rearrangement of this gives us 
    \begin{align*}
    d_\infty^2+d_\infty-4=0, 
    \end{align*}
    which contradicts $d_{\infty} \in \mathbf{Z}_{\geq 0}$. 
    Therefore, $M_{X} > 1$. 
    \item In the case $n\ge2$. \par
    Since $h$ is a strictly increasing function, $s = t$ is the only possibility for $s, t \in [0, 1]$ satisfying $(h(s)-h(t))(s-t)=0$. 
    Hence the set on which the integrand vanishes has measure zero in $[0,1]\times[0,1]$, and therefore 
    \begin{align*}
   &B(d_0+1,d_\infty+3)K-B(d_0+2,d_\infty+3)L \\
   &= \frac{1}{2}\int_0^1\int_0^1(h(s)-h(t))(s-t)g(s)g(t)dsdt > 0
 \end{align*} 
and hence 
    \[
    \frac{S_1}{S_0}=\frac{nK}{aL+K}>\frac{n(d_0+1)}{a(d_0+d_\infty+4)+(d_0+1)}\ge\frac{d_0+1}{d_\infty+1}. 
    \]
    This and the equivalence in \eqref{equiv_I_nneg_S_0_S_1} gives us that $M_{X} > 1$. 
\end{enumerate}
From the above, we see that $M_X>1$ holds in either case. 
\end{proof}

\subsection{The case $k \geq 2$ or $d_{\infty} \geq 1$}

We first show a non-existence result for Mabuchi solitons. 

\begin{proposition}
If $k\ge2$, or $k=1$ and $d_\infty\ge1$, then $M_{X} > 1$ holds, except in the cases $(n,k,d_0,d_\infty)=(1,1,0,1)$ and $(n,k,d_0,d_\infty)= (2,1,0,1)$. In particular, $X$ does not admit a Mabuchi soliton.
\end{proposition}
\begin{proof}
First note that the inequality \eqref{eq1} in Lemma \ref{inequity} is equivalent to 
  \begin{align}\label{eq2_eq1_equiv}
  \left[(d_\infty+1)(d_0+d_\infty+2)-\frac{d_0+d_\infty+4}{k}\right]n-\frac{d_0+d_\infty+4}{k}+2(d_0+1)\ge 0. 
  \end{align}
Define $f\colon[1,+\infty)\to\mathbf{R}$ by 
  \[
  f(t)\coloneq\left[(d_\infty+1)(d_0+d_\infty+2)-\frac{d_0+d_\infty+4}{k}\right]t-\frac{d_0+d_\infty+4}{k}+2(d_0+1). 
  \]
  \begin{enumerate}[(1)]
    \item The case $k\ge 2$. \par  
    In this case, we have 
    \begin{align*}
      &(d_\infty+1)(d_0+d_\infty+2)-\frac{d_0+d_\infty+4}{k} \\
      &\ge(d_\infty+1)(d_0+d_\infty+2)-\frac{d_0+d_\infty+4}{2} \\
      &=\left(d_\infty+\frac{1}{2}\right)(d_0+d_\infty+2)-1 \\
      &\ge 0. 
    \end{align*}
    Moreover, equality holds if and only if $(k,d_0,d_\infty)=(2,0,0)$. 
    \begin{enumerate}[(a)]
      \item The case $(k,d_0,d_\infty)=(2,0,0)$. \par
      Since 
        \[
          f(t)=\left(1\cdot2-\frac{4}{2}\right)t-\frac{4}{2}+2=0, 
        \]
        the inequality (\ref{eq1}) holds for any $n \in \mathbf{Z}_{\geq 1}$. 
        \item Cases other than (a). \par
        Since $k\ge2$, we have
        \begin{align*}
          f(1)&=(d_\infty+1)(d_0+d_\infty+2)-\frac{2(d_0+d_\infty+4)}{k}+2(d_0+1)\\
          &\ge(d_\infty+1)(d_0+d_\infty+2)-(d_0+d_\infty+4)+2(d_0+1)\\
          &=d_\infty(d_0+d_\infty+2)+2d_0 \\
          &\ge 0. 
        \end{align*}
        Therefore, since $n\ge2$ and $f$ is monotonically increasing, the inequality \eqref{eq1} is maintained by $f(n) \geq f(1)\ge 0$. 
    \end{enumerate}
      \item The case $k=1$. \par
      First note that 
      \[
        (d_\infty+1)(d_0+d_\infty+2)-(d_0+d_\infty+4)=d_\infty(d_0+d_\infty+2)-2>0
      \]
      and hence $f$ is a monotonically increasing function. 
      \begin{enumerate}[(a)]
        \item The case $d_0\ge 1$. \par 
      Since 
      \begin{align*}
        f(n) > f(1)&=(d_\infty+1)(d_0+d_\infty+2)-2(d_0+d_\infty+4)+2(d_0+1)\\
        &=(d_\infty-1)(d_0+d_\infty+2)+2(d_0-1)\ge 0, 
      \end{align*}
      the inequality (\ref{eq1}) holds. 
      \item The case $d_0=0$. \par
      If $d_\infty\ge 2$, then the inequality \eqref{eq1} holds from $f(1)\ge 3\cdot(0+2+2)-2>0$. 
      If $d_\infty=1$, then 
      \[
      f(3)=3\left((1+1)(0+1+2)-\frac{0+1+4}{1}\right)-\frac{0+1+4}{1}+2(0+1)=0, 
      \]
      and so, except in the cases $(n,d_0,d_\infty)=(1,0,1)$ and $(n,d_0,d_\infty)=(2,0,1)$, the inequality \eqref{eq1} holds. 
      \end{enumerate}
  \end{enumerate}
Therefore, if $k\ge2$, or $k=1$ and $d_\infty\ge1$, then $M_{X} > 1$ holds, except in the cases $(n,k,d_0,d_\infty)=(1,1,0,1)$ and $(n,k,d_0,d_\infty) = (2,1,0,1)$. In particular, $X$ does not admit a Mabuchi soliton.
\end{proof}

\begin{proposition}
If $(n,k,d_0,d_\infty)=(1,1,0,1)$, then $X$ admits a Mabuchi soliton. 
On the other hand, when $(n,k,d_0,d_\infty)=(2,1,0,1)$, $M_{X} > 1$ and hence $X$ does not admit a Mabuchi soliton. 
\end{proposition}

\begin{proof}
First, assume that $(n,k,d_0,d_\infty)=(1,1,0,1)$. 
Then, since 
    \[
    \lambda=\frac{2(1+1)/1+1-0}{0+1+2}=\frac{5}{3},\quad w=\frac{0-1}{0+1+2}=-\frac{1}{3}, 
    \]
we have 
    \begin{align*}
      I&=\int_{-1}^1\left(\frac{5}{3}+x\right)(1-x)^2\left(x+\frac{1}{3}\right)dx
      =\int_{-1}^1\left(x^4-\frac{22}{9}x^2+\frac{8}{9}x+\frac{5}{9}\right)dx\\
      &=\left\lbrack\frac{1}{5}x^5-\frac{22}{27}x^3+\frac{5}{9}x\right\rbrack_{-1}^1=-\frac{16}{135}<0. 
    \end{align*}
In particular, $M_{X} < 1$ and $X$ admits a Mabuchi soliton by Corollary \ref{existence_adm_M-sol}. 
Next, if $(n,k,d_0,d_\infty)=(2,1,0,1)$, then, since 
    \[
    \lambda=\frac{2(2+1)/1+1-0}{0+1+2}=\frac{7}{3},\quad w=\frac{0-1}{0+1+2}=-\frac{1}{3}, 
    \]
we have 
    \begin{align*}
      I&=\int_{-1}^1\left(\frac{7}{3}+x\right)^2(1-x)^2\left(x+\frac{1}{3}\right)dx\\
      &=\int_{-1}^1\left(x^5+3x^4-2x^3-\frac{194}{27}x^2+\frac{91}{27}x+\frac{49}{27}\right)dx\\
      &=\left\lbrack\frac{3}{5}x^5-\frac{194}{81}x^3+\frac{49}{27}x\right\rbrack_{-1}^1 \\
      &=\frac{16}{405} \\
     &> 0. 
    \end{align*}
This shows $M_{X} > 1$ and that $X$ does not admit a Mabuchi soliton. 
\end{proof}

\subsection{The case $k=1$ and $d_\infty=0$}

Next we show a result of existence for Mabuchi solitons. 
The following proposition completes the proof of Theorem \ref{M-sol_adm/P^n}. 

\begin{proposition}
If $k=1$ and $d_\infty=0$, then $X$ admits a Mabuchi soliton. 
\end{proposition}
\begin{proof}
Let $a$ and $b$ be real numbers defined in the proofs of Proposition \ref{nonvanishing_Futaki} and Lemma \ref{inequity}, respectively. 
Then, by \eqref{I_expression} in Lemma \ref{inequity}, we have 
  \begin{align*}
    I&=2^{d_0+d_\infty+n+3}\int_0^1(a+u)^nu^{d_0}(1-u)^{d_\infty+1}(u-b)du\\
    &=2^{d_0+n+3}\int_0^1(a+u)^nu^{d_0}(1-u)(u-b)du, \\
    a&=\frac{(n+1)-(d_0+1)}{d_0+2}=\frac{n-d_0}{d_0+2},\quad b=\frac{d_0+1}{d_0+2}. 
  \end{align*}
Here, by setting  
  \[
  \alpha \coloneqq -\frac{1}{2(d_0+2)},\quad \beta \coloneqq -\frac{d_0+1}{2(d_0+2)}, 
  \]
we obtain  
  \[
  u^{d_0}(1-u)(u-b)=\alpha\left(\frac{d}{du}u^{d_0+1}(1-u)^2\right)+\beta u^{d_0}(1-u)^2. 
  \]
Therefore, by integration by parts, 
    \begin{align*}
    I&=2^{d_0+n+3}\int_0^1(a+u)^n\left[\alpha\Bigl(\frac{d}{du}u^{d_0+1}(1-u)^2\Bigr)+\beta u^{d_0}(1-u)^2\right]du\\
    &=2^{d_0+n+3}\left\{\Bigl\lbrack\alpha(a+u)^nu^{d_0+1}(1-u)^2\Bigr\rbrack_0^1-\int_0^1\alpha n(a+u)^{n-1}u^{d_0+1}(1-u)^2du\right\} \\
    &\quad +2^{d_0+n+3}\int_0^1\beta(a+u)^nu^{d_0}(1-u)^2du \\
    &=2^{d_0+n+3}\int_0^1(a+u)^{n-1}u^{d_0}(1-u)^2(-\alpha nu+\beta(a+u))du\\
    &=\frac{2^{d_0+n+2}}{d_0+2}\int_0^1(a+u)^{n-1}u^{d_0}(1-u)^2[(n-d_0-1)u-a(d_0+1)]du. 
  \end{align*}
Thus, when $n-d_0-1=0$, we have 
  \[
  I=-\frac{2^{d_0+n+2}}{d_0+2}\int_0^1(a+u)^{n-1}u^{d_0}(1-u)^2a(d_0+1)du\,<0
  \]
and hence $M_{X} < 1$, which shows that $X$ admits a Mabuchi soliton. 
From now on, assume $n-d_0-1>0$. 
If we set 
  \[
  c\coloneq\frac{a(d_0+1)}{n-d_0-1}=\frac{(n-d_0)(d_0+1)}{(n-d_0-1)(d_0+2)}, 
  \]
we have $c > 0$ and 
  \[
  I=\frac{2^{d_0+n+2}(n-d_0-1)}{d_0+2}\int_0^1(a+u)^{n-1}u^{d_0}(1-u)^2(u-c)\,du. 
  \]
Hence, by setting 
  \[
  J\coloneq\int_0^1(a+u)^{n-1}u^{d_0+1}(1-u)^2du,\quad K\coloneq\int_0^1(a+u)^{n-1}u^{d_0}(1-u)^2du, 
  \]
we have $I = J-cK$, and $I \leq 0$ (respectively, $I = 0$) if and only if 
  \begin{align}\label{eq_1_equiv_JK}
  \frac{J}{K} \le c\quad \left(\text{respectively, } \frac{J}{K} = c\right). 
  \end{align}
In what follows, we show that $\dfrac{J}{K} < c$. 
First, 
  \begin{align*}
    J
    &=\sum_{j=0}^{n-1}\binom{n-1}{j}a^{n-1-j}\int_0^1u^{j+d_0+1}(1-u)^2du\\
    &=\sum_{j=0}^{n-1}\binom{n-1}{j}a^{n-1-j}B(j+d_0+2,3)\\
    &=\sum_{j=0}^{n-1}\binom{n-1}{j}a^{n-1-j}\frac{2}{(j+d_0+2)(j+d_0+3)(j+d_0+4)}, \\
    K
    &=\sum_{j=0}^{n-1}\binom{n-1}{j}a^{n-1-j}\int_0^1u^{j+d_0}(1-u)^2du\\
    &=\sum_{j=0}^{n-1}\binom{n-1}{j}a^{n-1-j}B(j+d_0+1,3)\\
    &=\sum_{j=0}^{n-1}\binom{n-1}{j}a^{n-1-j}\frac{2}{(j+d_0+1)(j+d_0+2)(j+d_0+3)}. 
  \end{align*}
Hence, by setting 
  \[
  p_j\coloneq\sum_{j=0}^{n-1}\binom{n-1}{j}a^{n-1-j},\quad q_j\coloneq\frac{2}{(j+d_0+1)(j+d_0+2)(j+d_0+3)}
  \]
we have 
  \[
  J=\sum_{j=0}^{n-1}p_jq_j\frac{j+d_0+1}{j+d_0+4},\quad K=\sum_{i=0}^{n-1}p_iq_i. 
  \]
Next, if we set $r_j\coloneq\dfrac{1}{K}p_jq_j$, then $\sum\limits_{j=0}^{n-1}r_j=1$ and 
  \[
  \frac{J}{K} = \sum_{j=0}^{n-1}r_j\frac{j+d_0+1}{j+d_0+4}. 
  \]
Define $f\colon[0, +\infty) \to\mathbf{R}$ by $f(t)\coloneq\dfrac{t+d_0+1}{t+d_0+4}$. 
Then $f$ is a strictly increasing concave function. 
Hence we have 
  \[
  \frac{J}{K}=\sum_{j=0}^{n-1}r_jf(j)\le f\left(\sum_{j=0}^{n-1}r_jj\right). 
  \]
Here, we point out that 
  \[
    \sum_{j=0}^{n-1}r_jj
    \le\frac{\sum\limits_{j=0}^{n-1}p_jj}{\sum\limits_{i=0}^{n-1}p_i} 
  \]
holds. 
In fact, since $p_j \geq 0$ and $q_j$ is monotonically decreasing in $j$, we have 
  \begin{align*}
    \sum_{j=0}^{n-1}p_jq_jj\sum_{i=0}^{n-1}p_i-\sum_{j=0}^{n-1}p_jj\sum_{i=0}^{n-1}p_iq_i
    &=\sum_{j=0}^{n-1}\sum_{i=0}^{n-1}j(q_j-q_i)p_jp_i\\
    &=\sum_{i<j}j(q_j-q_i)p_jp_i+\sum_{j<i}j(q_j-q_i)p_jp_i\\
    &=\sum_{i<j}j(q_j-q_i)p_jp_i+\sum_{i<j}i(q_i-q_j)p_ip_j\\
    &=\sum_{i<j}(j-i)(q_j-q_i)p_jp_i 
    \le 0. 
  \end{align*}
Hence, 
  \[
    \sum_{j=0}^{n-1}r_jj
    = \frac{\sum\limits_{j=0}^{n-1}p_jq_jj}{\sum\limits_{i=0}^{n-1}p_iq_i}
    \le\frac{\sum\limits_{j=0}^{n-1}p_jj}{\sum\limits_{i=0}^{n-1}p_i}. 
  \] 
Furthermore, since 
  \begin{align*}
    &\sum_{i=0}^{n-1}p_i =\sum_{i=0}^{n-1}\binom{n-1}{j}a^{n-1-i}=(1+a)^{n-1}, \\
    &\sum_{j=0}^{n-1}p_jj =\sum_{j=0}^{n-1}\binom{n-1}{j}j\cdot a^{n-1-j}=(n-1)(1+a)^{n-2}
  \end{align*}
and $f$ is monotonically increasing, 
  \[
  \frac{J}{K}\le f\left(\sum_{j=0}^{n-1}r_jj\right)\le f\left(\frac{\sum\limits_{j=0}^{n-1}p_jj}{\sum\limits_{i=0}^{n-1}p_i}\right)=f\left(\frac{n-1}{1+a}\right). 
  \]
Thus, we obtain 
  \begin{align*}
    f\left(\frac{n-1}{1+a}\right)
    &=f\left(\frac{(d_0+2)(n-1)}{n+2}\right) \\
    &=\frac{(2d_0+3)n+d_0}{2(d_0+3)n + d_0+6}
  \end{align*}
and 
  \begin{align}\label{JK_inequality}
  \frac{J}{K} 
  \le\frac{(2d_0+3)n+d_0}{2(d_0+3)n + d_0+6}. 
  \end{align}
Finally, let us show 
  \begin{align}\label{JK_sufficient}
  \frac{(2d_0+3)n+d_0}{2(d_0+3)n+(d_0+6)}<c. 
  \end{align}
First, by a direct computation, the condition \eqref{JK_sufficient} is equivalent to 
  \begin{align*}
  d_0n^2 + (d_0^2+12d_0+12)n - 4d_0(d_0+1) > 0. 
  \end{align*}
Moreover, since $d_0+1 < n$, we have in fact 
  \begin{align*}
    d_0n^2 + (d_0^2+12d_0+12)n - 4d_0(d_0+1) 
    &> d_0n^2 + (d_0^2+12d_0+12)n -4d_0n \\
    &= n(d_0n+d_0^2+8d_0+12) \\
    &>0. 
  \end{align*}
Therefore, by \eqref{JK_sufficient} and \eqref{JK_inequality} we obtain 
  \[
  \frac{J}{K}\le\frac{(2d_0+3)n+d_0}{2(d_0+3)n+(d_0+6)} < c. 
  \]
Hence, by \eqref{eq_1_equiv_JK}, we obtain $I < 0$, namely $M_{X} < 1$. 
\end{proof}

\begin{remark}
We note that the manifold of the form 
\[
\mathbf{P}_{\mathbf{P}^{n}}(\mathcal{O}^{\oplus(d_{0}+1)} \oplus \mathcal{O}(k)^{\oplus(d_{\infty}+1)})
\]
considered in this section is toric. 
Yao gave a formula for the Mabuchi constant of toric Fano manifolds in \cite[Section 5.1]{Y22b} as follows. 
Let $P$ be its anticanonical moment polytope, and let $\theta_{P}$ denote the affine function on $P$ characterized by 
\[
-f(0) + \frac{1}{V_{P}}\int_{P}f(x)\,dx = \frac{1}{V_{P}}\int_{P}f(x)\theta_{P}(x)\,dx
\]
for every affine function $f$ on $P$, where $V_{P}$ is the volume of $P$. 
 The Mabuchi constant is then given by
 \[
 M_{X} = \max_{P}\theta_{P}. 
 \]
 Concerning the relation between Yao's formula and Theorem \ref{M-const_proj_bdle}, a key ingredient is the identity
 \begin{align*}
     \theta_P(\mu_\omega)=\frac{1}{2}\Pi_\omega^G(\Scal(\omega)-2n),
 \end{align*}
which holds for the normalized moment map $\mu_\omega$ described in Section \ref{sec:v-sol} (see also \cite[Section 2.1]{Y22b}). 
Then, it follows from Mabuchi's result (\cite[Theorem 9.1]{Ma21}) that
 \begin{align*}
     \theta_P(\mu_\omega)=\Pi_\omega^G(1-e^{h_\omega})
     =\alpha z+\beta.
 \end{align*}
Therefore,
\begin{align*}
    \max_P\theta_P = \max_X\theta_P(\mu_\omega) 
    =\max_X(\alpha z+\beta),
\end{align*}
 which clarifies the relation between them.
Using Proposition \ref{prop_chara_poly} and \ref{mommap_normalization}, the explicit formula for the Mabuchi constant can then be recovered.
 
However, we emphasize that Theorem \ref{M-const_proj_bdle} itself applies to arbitrary Fano admissible manifolds, which are not necessarily toric. Thus the theorem, in its full generality, is not a direct consequence of Yao's result. 
\end{remark}



\begin{thebibliography}{99}

\bibitem{ACGTF08-3}
V.~Apostolov, D.~M.~J.~Calderbank, P.~Gauduchon and C.~T\o nnesen-Friedman, 
Hamiltonian 2-forms in K\"ahler geometry III: Extremal metrics and Stability. 
Invent. math. \textbf{173} No 3 (2008), 547--601. 

\bibitem{ACGTF08-4}
V.~Apostolov, D.~M.~J.~Calderbank, P.~Gauduchon and C.~T\o nnesen-Friedman,
Hamiltonian 2-forms in K\"ahler geometry IV: Weakly Bochner-flat K\"ahler manifolds. 
Comm. Anal. Geom. \textbf{16} (2008), 91--126. 

\bibitem{C82}
E~Calabi, Extremal K\"ahler metrics.
Seminar on {D}ifferential {G}eometry, Ann. of Math. Stud., vol. No. 102, Princeton Univ. Press, Princeton, NJ, 1982, pp.~259--290.

\bibitem{D17}
S.~K.~Donaldson,
The Ding functional, Berndtsson convexity and moment maps.
Geometry, analysis and probability, Progr. Math., vol. 310,
Birkh\"auser/Springer, Cham, 2017, pp.~57--67.

\bibitem{FM95} A.~Futaki and T.~Mabuchi, 
Bilinear forms and extremal K\"ahler vector fields associated with K\"ahler classes. 
Math. Ann. \textbf{301} (1995), 199--210. 

\bibitem{HL23}
J.~Han and C.~Li, 
On the Yau-Tian-Donaldson conjecture for generalized K\"ahler-Ricci soliton equations. 
Comm. Pure Appl. Math. \textbf{76} (2023), no. 9, 1793--1867. 

\bibitem{Hi23}
T.~Hisamoto, 
Mabuchi's soliton metric and relative D-stability.  
Amer. J. Math. \textbf{145} (2023), no. 3, 765--806.

\bibitem{Ma01}
T.~Mabuchi, 
K\"ahler-Einstein metrics for manifolds with nonvanishing Futaki character. 
Tohoku Math. J. (2) \textbf{53} (2001), no.2, 171--182. 

\bibitem{Ma21}
T. Mabuchi, Test configurations, stabilities and canonical K\"ahler metrics-complex geometry by the energy method.
SpringerBriefs Math. Springer, Singapore (2021). 

\bibitem{MTF11}
G.~Maschler and C.~T\o nnesen-Friedman, 
Generalizations of K\"ahler-Ricci solitons on projective bundles. 
Math. Scand. \textbf{108} (2011), no. 2, 161--176.
 
\bibitem{NN24} Y.~Nakagawa and S.~Nakamura, 
Multiplier Hermitian-Einstein metrics on Fano manifolds of KSM-type. 
Tohoku Math. J. (2) \textbf{76} (2024), no.1, 127--152. 

\bibitem{Nak19}
S.~Nakamura, 
Generalized K\"ahler Einstein metrics and uniform stability for toric Fano manifolds. 
Tohoku Math. J. (2) \textbf{71} (2019), no. 4, 525--532. 

\bibitem{Y22a} 
Y.~Yao, 
Relative Ding stability and an obstruction to the existence of Mabuchi solitons. 
J. Geom. Anal., (2022), no. 4, Paper number 105, 51 pp. 

\bibitem{Y22b}
Y.~Yao, 
Mabuchi solitons and relative Ding stability of toric Fano varieties. 
Int. Math. Res. Not. IMRN, (2022), no. 24, 19790--19853. 

\end{thebibliography}
\end{document}